\documentclass[12pt]{article} 
\usepackage[left=80pt, right=80pt, top=60pt]{geometry}

\usepackage[affil-it]{authblk} 
\usepackage{amsfonts}
\usepackage{enumerate} 
\usepackage{amssymb}
\usepackage{xfrac}	
\usepackage{slashed} 
\usepackage{color}
\usepackage{xcolor}
\usepackage{amsmath}
\usepackage[bookmarksopen,bookmarksdepth=2]{hyperref}
\providecommand{\noopsort}[1]{} 
\usepackage{authblk} 

\mathchardef\ordinarycolon\mathcode`\:
\mathcode`\:=\string"8000
\begingroup \catcode`\:=\active
  \gdef:{\mathrel{\mathop\ordinarycolon}}
\endgroup

\def\la{{\lambda}}

\newcommand{\N}{\mathbb N}
\newcommand{\Z}{\mathbb Z}

\newcommand{\R}{\mathbb R}
\newcommand{\C}{\mathbb C}

\renewcommand{\S}{\mathcal{S}}

\renewcommand{\H}{\mathcal{H}}
\renewcommand{\L}{\mathcal{L}}

\newcommand{\mB}{\mathcal{B}}
\newcommand{\W}{\mathcal{W}}

\renewcommand{\O}{\mathcal{O}}

\def\Re{{\mathrm{Re}\,}}

\newcommand{\loc}{\textnormal{loc}}
\newcommand{\oV}{\overline{V}}

\newcommand{\supnorm}[1]{\norm{ #1 }_\infty}
\newcommand{\absnorm}[1]{\norm{ #1 }_1}

\newcommand{\norm}[1]{\left\| #1 \right\|}
\newcommand{\nrm}[2]{\left\|#1\right\|_{#2}}

\newcommand{\Tr}{\operatorname{Tr}}

\usepackage{amsthm}
\newtheorem{thm}{Theorem}[section]
\newtheorem{lem}[thm]{Lemma}
\newtheorem{prop}[thm]{Proposition}
\newtheorem{cor}[thm]{Corollary}
\newtheorem{defn}[thm]{Definition}

\newtheorem{exmp}[thm]{Example}
\newtheorem{remark}[thm]{Remark}

\author{Teun D. H. van Nuland and Anna Skripka}
\begin{document}
\title{Higher-order spectral shift function for resolvent comparable perturbations}
\date{\today}


\maketitle
\let\thefootnote\relax
\let\thefootnote\relax



\begin{abstract}
Given a pair of self-adjoint operators $H$ and $V$ such that $V$ is bounded and $(H+V-i)^{-1}-(H-i)^{-1}$ belongs to the Schatten-von Neumann ideal $\mathcal{S}^n$, $n\ge 2$, of operators on a separable Hilbert space, we establish higher order trace formulas for a broad set of functions $f$ containing several major classes of test functions and also establish existence of the respective locally integrable real-valued spectral shift functions determined uniquely up to a low degree polynomial summand. Our result generalizes the result of \cite{PSS13} for Schatten-von Neumman perturbations $V$ and settles earlier attempts to encompass general perturbations with Schatten-von Neumman difference of resolvents, which led to more complicated trace formulas for more restrictive sets of functions $f$ and to analogs of spectral shift functions lacking real-valuedness and/or expected degree of uniqueness. Our proof builds on a general change of variables method derived in this paper and significantly refining those appearing in \cite{vNS21,PSS15,S17} with respect to several parameters at once.
\end{abstract}

\section{Introduction}
In many examples in mathematics and physics one is interested in a variational behavior of an operator functional of the form
\begin{align}\label{eq:operator trace functional}
	V\mapsto ``\Tr(f(H+V))",
\end{align}
where $H$ is a self-adjoint unbounded operator, $V$ a bounded operator called perturbation, and $f:\R\to\C$ a test function typically used to ensure the well-definedness either of the right-hand side of \eqref{eq:operator trace functional}, or of its respective Taylor remainder. Usually the function $f$ plays a secondary role, while important, physically relevant information is encoded in the 
behavior of the spectrum of $H+V$ as the perturbation $V$ varies.
The difference between the spectra of the operators $H$ and $H+V$ is captured by the \textit{first-order spectral shift function}, which is
a real-valued locally integrable function $\eta_1$ determined by $H$ and $V$ and satisfying
\begin{align*}
\Tr(f(H+V)-f(H))=\int_\R f'(x)\eta_1(x)\,dx.
\end{align*}
More delicate information on the motion of the spectrum of $H+V$ is captured by the \textit{$n^\text{th}$-order spectral shift function} (see, e.g., \cite{SZ20}), 
which is a real-valued locally integrable function $\eta_n$ determined by $H$ and $V$
 uniquely up to a certain degree and satisfying
\begin{align}
\label{tfintro}
	\Tr\Big(f(H+V)-\sum_{k=0}^{n-1}\frac{1}{k!}\frac{d^k}{dt^k}f(H+tV)\big|_{t=0}\Big)=\int_\R f^{(n)}(x)\eta_n(x)\,dx.
\end{align}
The spectral shift function $\eta_1$ emerged from the seminal mathematical and physical work of Krein \cite{Krein53} and Lifshits \cite{Lifshits}. Its generalization, the higher-order spectral shift function $\eta_n$, was suggested by Koplienko in \cite{Koplienko84}, and this function is the focus of this paper.

Since the introduction of the spectral shift function, there has been a slowly closing gap between theory and applications in terms of the admissible forms of $H$, $V$ and $f$. In many applications, neither the perturbation $V$ nor the resolvent $(H-i)^{-1}$ of $H$ is a compact operator. However, the property
\begin{align}
\label{rcc}
(H+V-i)^{-1}-(H-i)^{-1}\in\S^m,
\end{align}
which is known as the \textit{resolvent comparable condition}, can be satisfied. Examples of Dirac and Schr\"odinger operators satisfying \eqref{rcc} are given in \cite{S21} and their analogs in noncommutative geometry are discussed in \cite[Section 5]{vNS21} (see Remark \ref{rem:applications} for details).

\paragraph{Main result and potential applications.}

Let $\W^{n}_{k}$ denote the set of functions $f\in C^{n}$ satisfying $\widehat{(fu^l)^{(m)}}\in L^1$ for all $l=0,\ldots,k$ and $m=0,\ldots,n$, where $u(x)=x-i$. The set $\W^{n}_{k}$ includes all Schwartz functions and bounded rational functions that decay sufficiently fast (see Example \ref{wex} for details).

In our main result (Theorems \ref{main} and \ref{main even}) we establish the trace formula \eqref{tfintro} in the general case \eqref{rcc}. The respective spectral shift function is integrable with a smaller weight in the even-order case.
These two results are stated below.

\begin{thm}
\label{main in intro}
Let $n\in\N$, $n\ge 2$,  $H$ be a self-adjoint operator in $\H$, and let $V\in\mB(\H)_{\text{sa}}$ satisfy $(H+V-i)^{-1}-(H-i)^{-1}\in\S^n$.
Then, there exists a real-valued function $\eta_{2n-1}\in L^1(\R,u^{-4n-2}(x)dx)$
$($called the spectral shift function$)$ such that
\begin{align}
\label{tf2n-1 in intro}
\Tr\Big(f(H+V)-\sum_{k=0}^{2n-2}\frac{1}{k!}\frac{d^k}{dt^k}f(H+tV)\big|_{t=0}\Big)
=\int f^{(2n-1)}(x)\eta_{2n-1}(x)dx
\end{align}
for every $f\in\W^{2n-1}_{4n+2}$, and such that
\begin{align}\label{eq:main bound SSF in intro}
\int\frac{|\eta_{2n-1}(x)|}{(1+|x|)^{4n+2}}\,dx\leq c_{n}(1+\norm{V}^{2})\norm{V}^{n-1}\nrm{(H-i)^{-1}V(H-i)^{-1}}{n}^{n}
\end{align}
for a constant $c_{n}>0$.
The locally integrable function $\eta_{2n-1}$ is determined by \eqref{tf2n-1 in intro} uniquely up to a polynomial summand of degree at most $2n-2$.
\end{thm}
Similarly, we have the following result in the even-order case.
\begin{thm}
\label{main even in intro}
Let $n\in\N$, $n\ge 2$,  $H$ be a self-adjoint operator in $\H$, and let $V\in\mB(\H)_{\text{sa}}$ satisfy $(H+V-i)^{-1}-(H-i)^{-1}\in\S^n$.
Then, there exists a real-valued function $\eta_{2n}\in L^1(\R,u^{-4n-3}(x)dx)$ such that
\begin{align}\label{tf2n-1 even in intro}
\Tr\Big(f(H+V)-\sum_{k=0}^{2n-1}\frac{1}{k!}\frac{d^k}{dt^k}f(H+tV)\big|_{t=0}\Big)=\int f^{(2n)}(x)\eta_{2n}(x)dx
\end{align}
for every $f\in\W^{2n}_{4n+3}$ and such that
\begin{align*}
\int\frac{|\eta_{2n}(x)|}{(1+|x|)^{4n+3}}\,dx\leq c_{n}(1+\norm{V}^{2})\norm{V}^{n}\nrm{(H-i)^{-1}V(H-i)^{-1}}{n}^{n}
\end{align*}
for a constant $c_{n}>0$.
The locally integrable function $\eta_{2n}$ is determined by \eqref{tf2n-1 even in intro} uniquely up to a polynomial summand of degree at most $2n-1$.
\end{thm}

The aforementioned existence and integrability of the higher-order spectral shift functions satisfying the optimal trace formulas (i.e., the ones based on Taylor approximations of the perturbed operator function) are likely to give rise to results extending those known for the first-order spectral shift function. As outlined in \cite{BP}, the first-order spectral shift function has found many useful applications due to its connections with other important objects of mathematics. For instance, the relation between the regularized first-order perturbation determinant and the first-order spectral shift function was used in \cite{Yafaev07} to transfer results between the two objects in the setting of Schr\"{o}dinger operators with rapidly decaying potentials. The higher-order regularized perturbation determinant was related to the higher-order spectral shift function in \cite{Koplienko84}, although the existence of the latter had not been proven even under the more restrictive assumption $V\in\S^n$.
Since now we confirm the existence of the higher-order spectral shift functions in the general case, their connection with the perturbation determinant could be investigated for potential applications to new models.

We note that \eqref{rcc} is also a natural assumption in the case of an unbounded $V$ and that parts of this paper are generalizable to unbounded $V$. Thus, this paper might be an important step towards obtaining a higher-order spectral shift function for unbounded perturbations.

\paragraph{Novelty of our method.}

In order to derive our main results, we develop a new and rather general change of variables method for multilinear operator integrals, namely the one provided by the combination of Theorem \ref{thm:generalized change of variables} and Theorem \ref{thm:Extra expansion step}.

Theorem \ref{thm:generalized change of variables} unifies existing results in which a generic multilinear operator integral is expanded into a finite sum of multiple operator integrals whose inputs are multiplied by resolvents at specific locations, in practice improving their summability properties. In contrast to the change of variables formula derived in \cite{vNS21}, the locations where resolvents can be placed are not fixed by the new formula but instead can be flexibly adapted to the situation. In contrast to the change of variables formulas in \cite{PSS15,S17}, the operators supplying spectral data and the respective types of function classes are preserved.
Two significant advantages of the aforementioned flexibility and invariance are that they allow to
\begin{enumerate}[(i)]
\item enlarge the admissible class of functions and express it in terms of familiar function spaces,
\item extend trace formulas and properties of the multiple operator integral from the case where perturbations are Schatten class to the resolvent comparable case.
\end{enumerate}

In Theorem \ref{thm:Extra expansion step} we establish an estimate for the trace of a multilinear operator integral on tuples of noncompact resolvent comparable perturbations that is analogous to the best known estimate for the $\S^p$-norm of a multilinear operator integral on perturbations in $\S^q$, where $p,q>1$ (see \cite{PSS13}).
While the analytical result of Theorem \ref{thm:Extra expansion step} looks quite specific, its 
essence consists in circumventing the absence of good bounds for multilinear operator integrals in the trace class norm $(p=1)$ which is central to many applications of operator integrals.

Since Theorem \ref{thm:Extra expansion step} is proved under rather relaxed assumptions on perturbations and since already in this paper we need to apply the result of Theorem \ref{thm:generalized change of variables} several times for different purposes, we anticipate that both results will also find applications in other problems.

\paragraph{Comparison to previous results.}

Below we make a brief overview of most relevant prior results and state advantages of our new result.

The existence of a real-valued function $\eta_1\in L^1\big(\R,\frac{dx}{(1+|x|)^{1+\epsilon}}\big)$, $\epsilon>0$, satisfying \eqref{tfintro} with $n=1$ under the assumption \eqref{rcc} was settled in \cite[Theorem~3]{Krein62} 
and \cite[Section 8.8 (3)]{YGT}. The respective trace formula holds for a sufficiently large class of functions ensuring uniqueness of $\eta_1$ up to a constant summand.

The existence of the second-order spectral shift function
$\eta_2\in L^1\big(\R,\frac{dx}{(1+|x|)^{1+\epsilon}}\big)$, $\epsilon>0$, was established in \cite{Koplienko84} under the more restrictive assumption $V|H-i|^{-\frac12}\in\S^2$.
It was suggested in \cite{Koplienko84} and confirmed in \cite{PSS13} that a higher-order Taylor remainder can be expressed via a higher-order spectral shift measure in the case $V\in\S^n$. It was also established in \cite{PSS13} that those measures are absolutely continuous and, thus, described by a \textit{function}.

Many attempts were made to obtain \eqref{tfintro} under the assumption \eqref{rcc} for $m\ge 2$, but until now they resulted in much more complicated trace formulas for more restrictive sets of functions $f$ and the respective $\eta_n$ lacked real-valuedness and/or expected degree of uniqueness.

For instance, \cite[Theorem~3.2]{Neidhardt88} proves
that if $H$ and $V$ are self-adjoint operators satisfying \eqref{rcc} with $m=2$, then there exists a real-valued function
$\rho=\rho_{H,V}\in L^1\big(\R,\frac{dx}{(1+|x|)^2}\big)$ such that
\begin{align}
\label{trN}
\nonumber
&\Tr\left(\psi(H+V)-\psi(H)+(H-i)\frac{d}{dt}\psi(H+tM)\big|_{t=0}(H+i)\right)\\
&\quad=-\int_{-\infty}^\infty\rho(x)\frac{d}{dx}\big((1+x^2)\,\psi'(x)\big)\,dx
\end{align}
holds for all bounded rational functions $\psi$. Here $M$ is a bounded self-adjoint operator that equals $M=-\Re\big((H+V-i)^{-1}V(H+i)^{-1}\big)$ when $V$ is bounded.
Another trace formula with $n=2$ for $H,V$ satisfying \eqref{rcc} with $m=2$ under the additional assumption that $H$ and $H+V$ are bounded from below, without assuming that $V$ extends to a bounded operator, was obtained in \cite[Theorem 9.1]{GPS}.

It was established in \cite[Theorem 3.5]{PSS15} that if \eqref{rcc} is satisfied with $m\ge 2$, then there exists (a complex-valued function)
$\gamma_n=\gamma_{n,H,V}\in L^1\Big(\R,\frac{dx}{(1+|x|)^n}\Big)$ such that
\begin{align}
\label{trPSS}
\nonumber
&\Tr\bigg(\psi(H+V)-\psi(H)-\sum_{k=1}^{n-1}\,\sum_{\substack{j_1,\ldots,j_k\in\{1,\ldots,n-1\}\\
j_1<\ldots<j_k}}T_{\psi^{[k]}}^{H,\dots,H}(V_{j_1},V_{j_2-j_1},\ldots,V_{j_k-j_{k-1}})\bigg)\\
&=\frac{i^{n-1} }{2^{n-1}}\cdot\int_{-\infty}^\infty\,\frac{d^{n-1}}{dx^{n-1}}
\left((x-i)^n\psi'(x)\right)\gamma_n(x)\,dx
\end{align}
for all bounded rational functions $\psi$ with poles in the upper half-plane (but not in both half-planes), where
\begin{align*}
V_p=\big((I-V(H_1-i)^{-1})V(H_0-i)^{-1}\big)^{p-1}(I-V(H_1-i)^{-1})V\quad (p\in\N)
\end{align*}
and $T_{\psi^{[k]}}^{H,\dots,H}$ is a multilinear operator integral given by Definition \ref{nosep}.
The restriction on location of poles of $\psi$ in one half-plane was eliminated in \cite[Theorem 5.3, Corollary 5.4]{S17}, but the price to pay was an even more complicated trace formula.
%
More generally, the trace formulas \eqref{trN}, \eqref{trPSS}, and the one of \cite{S17} hold for every function $\psi$ of the form $\psi(\la)=\varphi\big(\frac{\la+i}{\la-i}\big)$, where
$\varphi(z)=\sum_{k=-\infty}^\infty a_kz^k$ and the Fourier series  of the respective derivative of $\varphi$ satisfies the following properties: $\sum_{k=-\infty}^\infty|a_k|k^2<\infty$ in the case of \eqref{trN}, $\sum_{k=0}^\infty|a_k|k^n<\infty$ and $a_{-k}=0$ for $k\in\N$ in the case of \eqref{trPSS}, $\sum_{k=-\infty}^\infty|a_k|k^n<\infty$ in the case of \cite{S17}, respectively.
However, no description of the set of admissible functions $\psi$ in terms of familiar function classes on $\R$ has been established.

The optimal trace formula \eqref{tfintro} for $n\ge 2$ with real-valued $\eta_n=\eta_{n,H,V}\in L^1\big(\R,\frac{dx}{(1+|x|)^{n+\epsilon}}\big)$, $\epsilon>0$, was established in \cite[Theorem 4.1]{vNS21} under the more restrictive assumption on the operators $V(H-i)^{-1}\in\S^n$. The respective trace formula holds, in particular, for all $(n+1)$-times continuously differentiable functions whose derivatives decay at infinity at the rate $f^{(k)}(x)=\O\left(|x|^{-k-\alpha}\right)$, $k=0,\ldots,n+1$, for some $\alpha>\frac12$. That class of functions determines $\eta_n$ uniquely up to a polynomial summand of degree at most $n-1$.
A variant of the trace formula was obtained in \cite[Theorem 3.8]{Yafaev07} for bounded rational functions of Schr\"{o}dinger operators with some rapidly decaying potentials. In that formula, the right-hand side of \eqref{tfintro} is replaced with $\int_\R f'(x)\eta_n(x)\,dx$ and $n$ depends on the dimension of the underlying Euclidean space.
\bigskip

Our main result has the following advantages over the results mentioned above:

\begin{enumerate}[(i)]
\item significant simplification of the earlier trace formulas under the general condition \eqref{rcc} for $m\ge 2$;
\item real-valuedness and uniqueness of the spectral shift function $\eta_n$ up to a polynomial summand of degree at most $n-1$;
\item description of the set of admissible functions in terms of familiar function spaces, which include bounded rational functions of sufficient decay and smooth compactly supported functions;
\item a relaxed assumption on $H$ and $V$, which is applicable to Schr\"odinger operators, Dirac operators, and generalized Dirac operators in noncommutative geometry.
\end{enumerate}

\section{Preliminaries}

In the sequel we fix a separable Hilbert space $\H$. Let $\mB(\H)$ denote the C*-algebra of bounded operators on $\H$, $\mB(\H)_\text{sa}$ its subset of self-adjoint operators, and $\norm{\cdot}$ the operator norm. We denote the Schatten $p$-class by $\S^p$ and its norm by $\nrm{\cdot}{p}$ for $p\in[1,\infty)$. For $p=\infty$ we use the convention that $\S^\infty=\mB(\H)$ and $\nrm{\cdot}{\infty}=\norm{\cdot}$.

We briefly write ``$H$ is self-adjoint in $\H$" when $H$ is a self-adjoint (possibly unbounded) operator densely defined in $\H$. We denote by $E_H$ the spectral measure of a self-adjoint $H$.

We denote $\N=\{1,2,\ldots\}$. We denote the cardinality of a set $A$ by $|A|$, a constant depending on parameters $a,b$ by $c_{a,b}$, and a constant depending on a tuple of parameters $(\alpha_1,\dots,\alpha_n)$ by $c_\alpha$.
The indices of a list $B_1,\ldots,B_k$ run upward if possible and the list is void if that is not possible. For instance, $(B_1,\ldots,B_k)$ stands for $(B_1)$ if $k=1$ and stands for $()$ if $k=0$.

We denote $i=\sqrt{-1}$ and
$$u(x):=x-i,$$
for all $x\in\R$, and write $u^k(x):=(u(x))^k$ for $k\in\Z$.

\paragraph{Function classes.}
Let $X$ be a metric space. We denote by $C(X)$ the space of functions continuous on $X$ and by $C_0(X),C_c(X),C_b(X),C^n(X)$, respectively, its subspaces of functions that are vanishing at infinity, compactly supported, bounded, $n$-times continuously differentiable. Let $C^n_c(X)=C^n(X)\cap C_c(X)$. Let $L^1(X)$ denote the space of Lebesgue integrable functions with associated norm denoted $\absnorm{\cdot}$, and $L^1_\loc(\R)$ the locally integrable functions on $\R$ with associated seminorms $f\mapsto\int_{-a}^a|f|$ for all $a>0$. When $X=\R$, the dependence on $X$ is sometimes suppressed in the notation.
Let $C^n_0(\R)$ denote the subset of $C^n$ of such $f$ for which $f^{(n)}\in C_0(\R)$.
We denote by $C_c^n(a,b)$ the functions in $C_c^n(\R)$ that are zero outside the interval $(a,b)$.

We define the Fourier transform of $f\in L^1(\R^n)$ by
	$$\hat{f}(x)=\int_{\R^n}f(y)e^{-iyx}\frac{dy}{(2\pi)^n}.$$
Whenever we state ``$\hat{f}\in L^1(\R)$'' for a continuous function $f$, we mean that $f$ is the (inverse) Fourier transform of a function in $L^1(\R)$, which, in particular, implies that $f\in C_0(\R)$ by the Riemann--Lebesgue lemma.

\begin{lem}\label{lem:tiny lem}
Let $n,k\in\N$, $a>0$, and $f\in C^n_c(-a,a)$. There exists a constant $c_{n,k,a}$ such that
	$$\supnorm{(fu^k)^{(p)}}\leq c_{n,k,a}\supnorm{f^{(n)}}$$
for every $p=0,\dots,n$.
\end{lem}
\begin{proof}
The lemma follows from the higher-order product rule for derivatives (Leibniz rule), the bound
$|f(x)|\le 2a\supnorm{f'}$ (by the mean value theorem)
and the fact that $|u(x)|\leq \sqrt{1+a^2}$ for all $x\in (-a,a)$.
\end{proof}

The following class of functions was introduced in \cite{vNvS21a}:
$$\W_{k}^{n}:=\{f\in C^n:~\widehat{(fu^l)^{(m)}}\in L^1\text{ for all $l=0,\ldots,k$ and $m=0,\ldots,n$}\}.$$
\begin{exmp}
\label{wex}
Examples of functions in $\W_k^n$ are functions in $C_c^{n+1}$ and Schwartz functions. Moreover, the function $x\mapsto \frac{1}{x-\lambda_1}\cdots\frac{1}{x-\lambda_{k+1}}$ is in $\W^n_k$ for any $\lambda_1,\ldots,\lambda_{k+1}\in\C\setminus\R$, and hence, any bounded rational function of order of decay $\O(|x|^{-k-1})$ is in $\W^n_k$.
\end{exmp}

We have the following useful properties of functions in $\W_{k}^n$.
\begin{prop}\label{prop:Wsn}
Let $k,n\in\N\cup\{0\}$. Then the following holds for all $f\in\W_k^n$.
\begin{enumerate}[(i)]
\item\label{prop:Wsn(i)} $f^{(m)}u^l\in C_0$ for all $m=0,\ldots,n$ and $l=0,\ldots,k$.
\item\label{prop:Wsn(ii)} $f^{(m)}u^l\in L^1$ for all $m=0,\ldots,n$ and $l=0,\ldots,k-2$.
\end{enumerate}
\end{prop}
\begin{proof}
By the remark on the Fourier transform made above, $(fu^l)^{(m)}\in C_0$ for all $m=0,\ldots,n$ and $l=0,\ldots,k$. In particular, $fu^l\in C_0$ for all $l=0,\ldots,k$. By induction on $m=0,\ldots,n$ and the higher-order Leibniz rule, we obtain \eqref{prop:Wsn(i)}.
The property (ii) follows from (i) and the identity $f^{(m)}u^l=f^{(m)}u^{l+2}u^{-2}$.
\end{proof}

Let $f^{[n]}$ denote the $n$th divided difference of $f\in C^n$, that is, $f^{[0]}:=f$ and
$$f^{[n]}(\lambda_0,\ldots,\lambda_n):=\lim_{\lambda\to\lambda_0}\frac{f^{[n-1]}(\lambda,\lambda_2,\ldots,\lambda_{n})-f^{[n-1]}(\lambda_1,\lambda_2,\ldots,\lambda_{n})}{\lambda-\lambda_1}.$$

\paragraph{Measures and partial integration}
\begin{lem}\label{lem:partial integration}
Let $n,m\in\N\cup\{0\}$, and let $\mu$ be a finite complex Radon measure on $\R$. For every $\epsilon\in(0,1]$ and every $k\in\N\cup\{0\}$ there exists a complex Radon measure $\tilde\mu_{k,\epsilon}$ with $$\norm{\tilde\mu_{k,\epsilon}}\leq\absnorm{u^{-1-\epsilon}}
\norm{\mu}$$ and
\begin{align}
\label{lsint}
\int g^{(n)}u^m \,d\mu=\int g^{(n+k)}u^{m+k+\epsilon}d\tilde\mu_{k,\epsilon}
\end{align}
for all $g\in C^{n+k}$ satisfying $g^{(n+l)}u^{m+l}\in C_0$, $l=0,\ldots,k-1$ and $g^{(n+k)}u^{m+k-1}\in L^1$, in particular, for all $g\in\W^{n+k}_{m+k+1}$.
\end{lem}

\begin{proof}
We shall assume that $k\geq1$ since the $k=0$ case is trivial. Define
$$\xi_1(x):=\int_0^x u^m(y)\,d\mu(y)$$
We have
$$\|{u^{-m}\xi_1}\|_\infty\leq \sup_{x\in\R}\big(|u^{-m}(x)|\sup_{|y|\leq|x|}|u^m(y)|\norm{\mu}\big)
=\sup_{x\in\R}\big(|u^{-m}(x)|\cdot|u^m(x)|\cdot\norm{\mu}\big)=\norm{\mu}.$$
By partial integration for all $a>0$,
\begin{align*}
\int_{-a}^a g^{(n)}u^m\,d\mu &=\big[g^{(n)}u^m{u^{-m}\xi_1}\big]_{-a}^a-\int_{-a}^a g^{(n+1)}(x){\xi_1}(x)dx.
\end{align*}
Since $g^{(n)}u^m\in C_0$, $g^{(n+1)}u^m\in L^1(\R)$, and $u^{-m}\xi_1\in C_b$, taking the limit above as $a\to\infty$ yields
\begin{align*}
\int g^{(n)}u^m\,d\mu &=-\int g^{(n+1)}(x)\xi_1(x)dx.
\end{align*}
We define, recursively, $\xi_2(x):=\int_0^x\xi_1(y)\,dy,\;\ldots,\;\xi_k(x):=\int_0^x\xi_{k-1}(y)\,dy$.
By an argument similar to the one above,
$$\supnorm{u^{-m-k+1}\xi_k}\leq\supnorm{u^{-m-k+2}\xi_{k-1}}\leq\cdots
\leq\supnorm{u^{-m}\xi_1}\leq\norm{\mu}$$ and
\begin{align*}
\int g^{(n)}u^m\,d\mu=(-1)^k\int g^{(n+k)}(x)\xi_k(x)dx
\end{align*}
for $g$ satisfying $g^{(n+l)}u^{m+l}\in C_0$, $l=0,\ldots,k-1$ and $g^{(n+k)}\in L^1$.
Thus, \eqref{lsint} holds with
$$d\tilde\mu_{k,\epsilon}(x)=(-1)^ku^{-m-k-\epsilon}(x)\xi_k(x) dx,$$ where $\norm{\tilde\mu_{k,\epsilon}}=\absnorm{u^{-1-\epsilon}u^{-m-k+1}\xi_k} \leq\absnorm{u^{-1-\epsilon}}\norm{\mu}$.
\end{proof}

\paragraph{Schatten classes and strong operator topology.}
Although the following lemma is known, we provide a short proof for convenience of the reader.
\begin{lem}\label{lem:sequential continuity}
Define the topological spaces $\L^p=(\S^p,\nrm{\cdot}{p})$ for every $p\in[1,\infty)$ and $\L^\infty=(\mB(\H),\textnormal{so*})$. Let $n\in\N$ and let $\alpha,\alpha_1,\ldots,\alpha_n\in[1,\infty]$ be such that $\frac{1}{\alpha_1}+\ldots+\frac{1}{\alpha_n}=\frac{1}{\alpha}$. Then, the function
\begin{align*}
(A_1,\ldots,A_n)&\mapsto A_1\cdots A_n
\end{align*}
is a sequentially continuous map from $\L^{\alpha_1}\times\cdots\times\L^{\alpha_n}$ to $\L^{\alpha}$.
\end{lem}
\begin{proof}
By using induction it suffices to prove the lemma for $n=2$. That is, it suffices to prove that the multiplication map from $\L^p\times\L^q$ to $\L^r$ is sequentially continuous for $p,q,r\in [1,\infty]$ satisfying $\frac{1}{p}+\frac{1}{q}=\frac{1}{r}$.

If $p,q<\infty$, then H\"older's inequality for noncommutative $L^p$-spaces (cf. \cite[Theorem 4.2]{FK}) implies continuity of the multiplication map $\L^p\times\L^q\to\L^r$.

Suppose that $p=\infty$, $q<\infty$. If $b\in\S^q$, then the map $\L^\infty\to\L^q$, $a\mapsto ab$ is continuous by approximating $b$ by rank-one operators. Joint continuity of the multiplication map $\L^\infty\times\L^q\to\L^q$ follows by an $\epsilon/2$-argument, because a converging sequence in $\L^\infty$ is bounded.

By taking adjoints, continuity of the multiplication map $\L^p\times\L^\infty\to\S^p$ follows for $p<\infty$.

If $p=q=\infty$, then the claim follows from the fact that multiplication on the unit ball of $\mB(\H)$ is jointly strongly continuous.
\end{proof}

\paragraph{Multiple operator integrals.}

We use the following definition of the multilinear operator integral, introduced in \cite{PSS13} and extending the definitions from \cite{azamov09} and \cite{Peller} (cf. \cite[Definition 4.3.3]{ST19}).

\begin{defn}
\label{nosep}
Let $n\in\N$, let $H_0,\dots,H_n$ be self-adjoint in $\H$, let $\phi:\R^{n+1}\to\C$ be bounded Borel, and let $\alpha,\alpha_1,\ldots,\alpha_n\in [1,\infty]$ satisfy $\tfrac{1}{\alpha}=\tfrac{1}{\alpha_1}+\ldots+\tfrac{1}{\alpha_n}$.
Denote $E^j_{l,m}:=E_{H_j}\big(\big[\frac{l}{m},\frac{l+1}{m}\big)\big)$.
If for all $V_j\in\S^{\alpha_j}$, $j=1,\ldots,n$,
the iterated limit
$$T^{H_0,\ldots,H_n}_{\phi}(V_1,\ldots,V_n)
:=\lim_{m\to\infty}\lim_{N\to\infty}\sum_{|l_0|,\ldots,|l_n|<N}
\phi\left(\frac{l_0}{m},\ldots,\frac{l_n}{m}\right)E^0_{l_0,m}V_1E^1_{l_1,m}\cdots V_nE^n_{l_n,m}\,$$
exists in $\S^\alpha$, then the transformation $T^{H_0,\ldots,H_n}_{\phi}:\S^{\alpha_1}\times\cdots\times\S^{\alpha_n}\to\S^\alpha$, which is a bounded operator by the uniform boundedness principle, is called a multilinear operator integral (MOI). If this is the case, we write $T_\phi^{H_0,\ldots,H_n}\in\mB(\S^{\alpha_1}\times\cdots\times\S^{\alpha_n},\S^\alpha)$.
\end{defn}

We will also consider a degenerate multilinear operator integral $T^{H_0}_{\phi}()=\phi(H_0)$.

The following theorem, proven in \cite[Lemma 3.5]{PSS13}, shows that, if the symbol $\phi$ admits a separation of variables, the multiple operator integral coincides with the one from \cite{azamov09,Peller} and admits a strong bound.

\begin{thm}\label{thm:iptp}
Let $n\in\N$ and let $H_0,\ldots,H_n$ be self-adjoint in $\H$, and let $(\Omega,\nu)$ be a finite measure space. Suppose that $a_j(\cdot,s):\R\to\C,\quad s\in\Omega,$ are bounded continuous functions and that there is a sequence $\{\Omega_k\}_{k=1}^\infty$ of growing measurable subsets
of $\Omega$ such that $\Omega=\cup_{k=1}^\infty\Omega_k$ and the families
$$\{a_j(\cdot,s)\}_{s\in\Omega_k},\quad j=0,\dots,n$$
are uniformly bounded and uniformly equicontinuous. Define $\phi:\R^{n+1}\to\C$ by
\begin{align}\label{phi representation}
	\phi(\lambda_0,\ldots,\lambda_n)=\int_\Omega a_0(\lambda_0,s)\cdots a_n(\lambda_n,s)\,d\nu(s).
\end{align}
Then we have
$T^{H_0,\ldots,H_n}_\phi\in\mB( \S^{\alpha_1}\times\cdots\times\S^{\alpha_n},\S^\alpha)$ for all $\alpha,\alpha_j\in[1,\infty]$ with $\tfrac{1}{\alpha_1}+\ldots+\tfrac{1}{\alpha_n}=\tfrac{1}{\alpha}$, as well as
\begin{align*}
T^{H_0,\ldots,H_n}_\phi(V_1,\ldots,V_n)(y)=\int_\Omega a_0(H_0,s)V_1 a_1(H_1,s)\cdots V_n a_n(H_n,s)y\,d\nu(s),\quad y\in\H,
\end{align*}
and
\begin{align*}
\big\|T_{\phi}^{H_0,\dots,H_n}\big\|_{\S^{\alpha_1}\times\cdots\times\S^{\alpha_n}\to\S^\alpha}
\le \int_\Omega\prod_{j=0}^n\|a_j(\cdot,s)\|_\infty\,d|\nu|(s).
\end{align*}
\end{thm}

The following property of the multiple operator integral is a consequence of \cite[Lemmas 3.5, 5.1 and 5.2]{PSS13} (see also \cite[Theorem 3.9]{vNS21} and \cite[Lemma 4.6]{azamov09}).
\begin{thm}
\label{Salphamember}
Let $n\in\N$ and let $f\in C^n(\R)$ be such that $\widehat{f^{(n)}}\in L^1(\R)$.
Let $\alpha,\alpha_1,\dots,\alpha_n\in[1,\infty]$ be such that $\frac{1}{\alpha}=\frac{1}{\alpha_1}+\dots+\frac{1}{\alpha_n}$.
Let $H_0,\dots,H_{n}$ be self-adjoint operators in $\H$ and let $V_k\in\S^{\alpha_k}$, $k=1,\dots,n$.
Then $\phi=f^{[n]}$ satisfies the assumptions of Theorem \ref{thm:iptp}
and
\begin{align*}
\|T_{f^{[n]}}^{H_0,\dots,H_n}(V_1,\dots,V_n)\|_{\alpha}
\le \frac{1}{n!}\,\|\widehat{f^{(n)}}\|_1\|V_1\|_{\alpha_1}\dots\|V_n\|_{\alpha_n},
\end{align*}
in particular, $T_{f^{[n]}}^{H_0,\dots,H_n}(V_1,\dots,V_n)\in\S^\alpha$.
\end{thm}

We derive the following lemma from the respective results of \cite[Section 4.3]{ST19}, which in turn were derived from the major results of \cite{PSS13}. The resulting lemma gives useful estimates and representations for traces of multiple operator integrals, improving upon \cite[Theorem 2.6 and Lemma 2.7]{CS18}.

\begin{lem}\label{lem43}
Let $k\in\N$ and let $H_0,\ldots,H_k$ be self-adjoint operators in $\H$.
\begin{enumerate}[(i)]
\item\label{item:tr bound i}
Let $\alpha_1\ldots,\alpha_k\in[1,\infty)$ be such that $\tfrac{1}{\alpha_1}+\ldots+\tfrac{1}{\alpha_k}=1$, and let $B_j\in\S^{\alpha_j}$, $j=1,\dots,k$.
If $H_0=H_k$, then there exists $c_\alpha=c_{\alpha_1,\dots,\alpha_k}>0$ such that
\begin{align*}
|\Tr(T^{H_0,\ldots,H_k}_{f^{[k]}}(B_1,\ldots,B_k))|&\leq c_\alpha\supnorm{f^{(k)}}\nrm{B_1}{\alpha_1}\cdots\nrm{B_k}{\alpha_k}\qquad (f\in C^k, \widehat{f^{(k)}}\in L^1)
\end{align*}
and there exist a unique $($complex$)$ Radon measure $\mu_1$ with total variation bounded by $c_\alpha\nrm{B_1}{\alpha_1}\cdots\nrm{B_k}{\alpha_k}$ such that
\begin{align*}
\Tr(T^{H_0,\ldots,H_k}_{f^{[k]}}(B_1,\ldots,B_k))&=\int_\R f^{(k)} d\mu_1\qquad (f\in C^k, \widehat{f^{(k)}}\in L^1).
\end{align*}

\item\label{item:tr bound ii}
Let $\alpha_1,\ldots,\alpha_k\in[1,\infty)$ be such that $\tfrac{1}{\alpha_1}+\ldots+\tfrac{1}{\alpha_k}=1$, and let $B_j\in\S^{\alpha_j}$, $j=1,\dots,k$.
Then, there exists $c_\alpha=c_{\alpha_1,\dots,\alpha_k}>0$ such that
\begin{align*}
|\Tr(B_1T^{H_1,\ldots,H_{k}}_{f^{[k-1]}}(B_2,\ldots,B_k))|&\leq c_\alpha\supnorm{f^{(k-1)}}\nrm{B_1}{\alpha_1}\cdots\nrm{B_k}{\alpha_k}\qquad ( f\in C_b^{k-1}).
\end{align*}
and there exist a unique Radon measure $\mu_2$ satisfying
$\|\mu_2\|\le c_\alpha\nrm{B_1}{\alpha_1}\cdots\nrm{B_k}{\alpha_k}$ and
\begin{align*}
\Tr(B_1 T^{H_1,\ldots,H_{k}}_{f^{[k-1]}}(B_2,\ldots,B_k))&=\int_\R f^{(k-1)}d\mu_2\,\qquad (f\in C_0^{k-1}).
\end{align*}
\end{enumerate}
\end{lem}

\begin{proof}
(i) For $f\in C^k$ with $\widehat{f^{(k)}}\in L^1$ we have $T_{f^{[k]}}^{H_0,\ldots,H_k}(B_1,\ldots,B_k)\in\S^1$ by Theorem \ref{Salphamember}. By Theorem \ref{thm:iptp}, we obtain
$$\Tr(T^{H_k,H_1,\ldots,H_k}_{f^{[k]}}(B_1,\ldots,B_k))
=\Tr(B_1T^{H_1,\ldots,H_k}_{\psi_{f,k}}(B_2,\ldots,B_k))$$
for $\psi_{f,k}(\lambda_1,\ldots,\lambda_k):=f^{[k]}(\lambda_k,\lambda_1,\ldots,\lambda_k)$. The symbol $\psi_{f,k}$ is an instance of a polynomial integral momentum (see \cite[Equation (4.3.14)]{ST19}) of the form $\phi_{k-1,f^{(k)},p}$ for $p(s_0,\ldots,s_{k-1})=s_0$, and therefore the corresponding multiple operator integral $T_{\psi_{f,k}}$ can be bounded by \cite[Theorem 4.3.10]{ST19} as follows. Assume firstly that $k\geq 2$. In that case, $\frac{1}{\alpha_1}+\ldots+\frac{1}{\alpha_k}=1$ implies that $\alpha_1,\ldots,\alpha_k\in(1,\infty)$. By H\"older's inequality and \cite[Theorem 4.3.10]{ST19}, we have
\begin{align*}
|\Tr(T^{H_k,H_1,\ldots,H_k}_{f^{[k]}}(B_1,\ldots,B_k))|&\leq \nrm{B_1}{\alpha_1}\nrm{T^{H_1,\ldots,H_k}_{\psi_{f,k}}(B_2,\ldots,B_k)}{\alpha_1'}\\
&\leq c_\alpha\supnorm{f^{(k)}}\nrm{B_1}{\alpha_1}\cdots\nrm{B_k}{\alpha_k},
\end{align*}
where $\alpha_1'\in(1,\infty)$ is the H\"older conjugate of $\alpha_1$. If $k=1$, then $\psi_{f,k}=f'$ and we find
\begin{align*}
|\Tr(T^{H_1,H_1}_{f^{[1]}}(B_1))|= |\Tr(B_1f'(H_1))|\leq \supnorm{f^{(1)}}\nrm{B_1}{1}.
\end{align*}
Therefore, the first statement of \eqref{item:tr bound i} follows. By the Riesz--Markov representation theorem, along with the Hahn--Banach theorem, we obtain the second statement of \eqref{item:tr bound i}.

Item \eqref{item:tr bound ii} follows completely analogously, by using the polynomial integral momentum $f^{[k-1]}=\phi_{k-1,f^{(k-1)},1}$.
\end{proof}

We will need the following perturbation formula stated in \cite[Theorem 3.3.8]{ST19} for $n=0$ and in \cite[Theorem 4.3.14]{ST19} for $n\in\N$.
\begin{thm}
\label{perturbation}
Let $n\in\N\cup\{0\}$ and let $f\in C^{n+1}(\R)$ be such that $\widehat{f^{(k)}}\in L^1(\R)$, $k=1,\dots,n+1$.
Let $H_1,\dots,H_n,A,B$ be self-adjoint in $\H$ such that $A-B$ is bounded.
Let $V_1,\dots,V_{n}$ be bounded operators on $\H$. Then, for every $i=1,\dots,n+1$,
\begin{align*}
&T_{f^{[n]}}^{H_1,\dots,H_{i-1},A,H_i,\dots,H_n}(V_1,\dots,V_n)
-T_{f^{[n]}}^{H_1,\dots,H_{i-1},B,H_i,\dots,H_n}(V_1,\dots,V_n)\\
&=T_{f^{[n+1]}}^{H_1,\dots,H_{i-1},A,B,H_i,\dots,H_n}(V_1,\dots,V_{i-1},A-B,V_i,\dots,V_n).
\end{align*}
\end{thm}

The following useful change of variables formula was obtained in \cite[Theorem 3.10]{vNS21}.
\begin{lem}\label{lem:change of variables}
	Let $n\in\N$, let $H_0,\ldots,H_n$ be self-adjoint in $\H$, let $V_1,\ldots,V_n\in\mB(\H)$, and let $f\in C^n$ be such that $\widehat{f^{(n-1)}},\widehat{f^{(n)}},\widehat{(fu)^{(n)}}\in L^1(\R)$.
Then, the following assertions hold.
\begin{enumerate}[(i)]
\item\label{cv1}	
\begin{align*}
T^{H_0,\ldots,H_n}_{f^{[n]}}(V_1,\ldots,V_n)=&(H_0-i)^{-1}T^{H_0,\ldots,H_n}_{(fu)^{[n]}}(V_1,\ldots,V_n)\\
		&-(H_0-i)^{-1}V_1T^{H_1,\ldots,H_n}_{f^{[n-1]}}(V_2,\ldots,V_n).
	\end{align*}

\item\label{cv2}
For every $j\in\{1,\ldots,n-1\}$,
\begin{align*}
T^{H_0,\ldots,H_n}_{f^{[n]}}(V_1,\ldots,V_n)=&T^{H_0,\ldots,H_n}_{(fu)^{[n]}}(V_1,\ldots,V_j(H_j-i)^{-1},V_{j+1},\ldots, V_n)\\
&-T^{H_0,\ldots,H_{j-1},H_{j+1},\ldots,H_n}_{f^{[n-1]}}(V_1,\ldots,V_j(H_j-i)^{-1}V_{j+1},\ldots,V_n).
	\end{align*}

\item\label{cv3}
\begin{align*}
T^{H_0,\ldots,H_n}_{f^{[n]}}(V_1,\ldots,V_n)=&T^{H_0,\ldots,H_n}_{(fu)^{[n]}}(V_1,\ldots,V_n)(H_n-i)^{-1}\\
		&-T^{H_0,\ldots,H_{n-1}}_{f^{[n-1]}}(V_1,\ldots,V_{n-1})V_n(H_n-i)^{-1}.
	\end{align*}
\end{enumerate}
\end{lem}

\paragraph{Taylor remainder.}
Let $R_{n,H,f}(V)$ denote the $n$th Taylor remainder of the approximation of $f(H+V)$ by $f(H)$, that is,
\begin{align*}
R_{n,H,f}(V)=f(H+V)-\sum_{k=0}^{n-1}\frac{1}{k!}\frac{d^k}{dt^k}f(H+tV)\big|_{t=0}.
\end{align*}

Multiple operator integrals are multilinear extensions of operator derivatives, in the sense of the following theorem, proven, e.g., in \cite[Theorem 5.3.5]{ST19}.

\begin{thm}
\label{dm}
Let $n\in\N$ and let $f\in C^n(\R)$ be such that $\widehat{f^{(k)}}\in L^1(\R)$, $k=1,\dots,n$. Let $H$ be self-adjoint in $\H$, let $V\in\mB(\H)_{\textnormal{sa}}$.
Then the Fr\'{e}chet derivative $\frac{1}{k!}\frac{d^k}{dt^k}f(H+tV)|_{t=0}$ exists in the operator norm and equals the multiple operator integral
\begin{align}
\label{dermoi}
\frac{1}{k!}\frac{d^k}{ds^k}f(H+sV)\big|_{s=t}=T_{f^{[k]}}^{H+tV,\dots,H+tV}(V,\dots,V).
\end{align}
The map $t\mapsto\frac{d^k}{ds^k}f(H+sV)|_{s=t}$
is strongly continuous and, when $V\in\S^n$, $\S^1$-continuous.
\end{thm}

We will use the following representation of the Taylor remainder via the multiple operator integral,
which is obtained by a recursive application of Theorem \ref{perturbation} and representation \eqref{dermoi} for $t=0$.
\begin{thm}
\label{rm}
Let $n\in\N$ and let $f\in C^n(\R)$ be such that $\widehat{f^{(k)}}\in L^1(\R)$, $k=1,\dots,n$.
Let $H$ be self-adjoint in $\H$ and let $V\in\mB(\H)_{\text{sa}}$. Then,
\begin{align}
\label{remmoi}
R_{n,H,f}(V)=T^{H,H+V,H,\ldots,H}_{f^{[n]}}(V,\ldots,V).
\end{align}
\end{thm}

\paragraph{Resolvent comparable condition.}

\begin{lem}\label{lem:resolvent comparable}
Let $H$ be self-adjoint in $\H$ and let $V\in\mB(\H)_\text{sa}$.
The condition $$(H+V-i)^{-1}-(H-i)^{-1}\in\S^n$$ is equivalent to the condition $$(H-i)^{-1}V(H-i)^{-1}\in\S^n.$$
\end{lem}
\begin{proof}
The second resolvent formula, applied twice, gives
\begin{align}
	&(H-i)^{-1}-(H+V-i)^{-1}\label{rc1}\\
	&=(H-i)^{-1}V(H-i)^{-1}-\big((H-i)^{-1}-(H+V-i)^{-1}\big)V(H-i)^{-1}\label{rc2}\\
	&=(H-i)^{-1}V(H-i)^{-1}-(H+V-i)^{-1}V(H-i)^{-1}V(H-i)^{-1}\label{rc3}
\end{align}
The equality \eqref{rc1}=\eqref{rc2} shows that $(H+V-i)^{-1}-(H-i)^{-1}\in\S^n$ implies that $(H-i)^{-1}V(H-i)^{-1}\in\S^n$, whereas the equality \eqref{rc1}=\eqref{rc3} shows the reverse implication.
\end{proof}

\begin{remark}\label{rem:applications}
(i) Lemma \ref{lem:resolvent comparable}, in particular, shows that in either of the cases
$$V\in\S^n,\qquad (H-i)^{-1}\in\S^n,\qquad\text{or}\qquad V(H-i)^{-1}\in\S^n, $$
the resolvent comparable condition \eqref{rcc} is satisfied. The first case is quite stringent but is applicable to some discrete models.
The second and third cases are applicable to perturbations of differential operators (such as Dirac and Schr\"odinger operators) on respectively compact and locally compact spaces under certain restrictions on the dimension of the underlying space and function class of the potential (see \cite[Section 5.2]{vNS21} for details).
Generalizations of differential operators arising in noncommutative geometry also satisfy 
the condition \eqref{rcc} (see \cite[Section 5.1]{vNS21} for details).

(ii) To consider specific examples of operators satisfying \eqref{rcc}, let $\Delta$ denote the Laplacian densely defined in $L^2(\R^d)$, let $D$ denote the Dirac operator on $\C^{2^{\lfloor(d+1)/2\rfloor}}\otimes L^2(\R^d)$, and let $V$ denote the operator of multiplication by a real-valued function $v\in L^p(\R^d)\cap L^\infty(\R^d)$, where $d\in\N$ and $p\in[1,\infty)$. The following sufficient conditions are verified in \cite[Theorem III.1]{S21}:
\begin{align*}
&\text{if }\;p>\frac{d}{4},\;\text{ then }\;(-\Delta+V-i)^{-1}-(-\Delta-i)^{-1}\in \S^p;\\
&\text{if }\;p>\frac{d}{2},\;\text{ then }\;(-D+V-i)^{-1}-(-D-i)^{-1}\in \S^p.
\end{align*}
Suitable perturbations of massive Dirac operators with electromagnetic potentials also satisfy the condition \eqref{rcc}. Specific examples of such operators are discussed in
\cite[Remark III.6(iii)]{S21} and \cite[Remark 5.4]{vNS21}.

%
%
\end{remark}

We will need the following useful property of a perturbed resolvent.
\begin{lem}
\label{ric}
Let $H$ be self-adjoint in $\H$ and let $V\in \mB(\H)_{\text{sa}}$. Then,
\begin{align*}
&(H+V-i)^{-1}=(I-(H+V-i)^{-1}V)(H-i)^{-1},\\
&(H+V-i)^{-1}=(H-i)^{-1}(I-V(H+V-i)^{-1}).
\end{align*}
\end{lem}
\begin{proof}
Both identities follow from the second resolvent identity.
\end{proof}

Given self-adjoint operators $H_0,\dots,H_n$ in $\H$ (some of which can possibly coincide) and $V_1,V_2,\ldots\in\mB(\H)$ (some of which can possibly coincide), we denote
\begin{align}
\label{bardef}
\nonumber
&\overline{V}_1=(H_0-i)^{-1}V_1(H_1-i)^{-1},\quad\overline{V}_2=V_2,\quad\overline{V}_3=(H_2-i)^{-1}V_3(H_3-i)^{-1},\dots,\\
&\overline{V}_{i,j}=\overline{V}_{i+1}\cdots\overline{V}_{j},\\
\nonumber
&\overline{V}_{0,0}=\oV_{n,n}=I.
\end{align}

\section{Estimating the trace of MOI on noncompact perturbations.}

In this section we develop a method enabling estimates for multilinear operator integrals on tuples of 
resolvent comparable
perturbations that are analogous to the estimates of Lemma \ref{lem43} for multilinear operator integrals on perturbations in $\S^\alpha$ for $1<\alpha<\infty$.
The method builds on a two-stage change of variables. The first stage is a generalization of the change of variables formula obtained in \cite{vNS21}. The second stage is an analytical result aimed specifically at resolvent comparable perturbations, which is derived from the first stage, but is also designed to be applied alongside the first stage.

\paragraph{Change of variables to confirm trace class membership of MOI.}

Let $m\in\N$ and let $H_0,\dots,H_m$ be self-adjoint in $\H$.
We will use formal symbols $\text{L}$ and $\text{R}$ to indicate whether a resolvent $(H_j-i)^{-1}$ is placed on the left of a comma, or on the right of a comma, respectively.
For $U_0,\ldots,U_m\in\mB(\H)$, $\epsilon=(\epsilon_0,\ldots,\epsilon_m)\in\{\text{L}, 0,\text{R}\}^{m+1}$ and $j\in\{0,\ldots,m\}$, we define $\epsilon_{-1}=0$ and
\begin{align}
\label{hatU}
	\check{U}^\epsilon_j:=
	\begin{cases}
	(H_{j-1}-i)^{-1}U_j(H_j-i)^{-1}&\text{if } \epsilon_{j-1}=\text{R},~\epsilon_j=\text{L}\\
		U_j(H_j-i)^{-1}&\text{if } \epsilon_{j-1}\neq\text{R},~\epsilon_j=\text{L}\\
			(H_{j-1}-i)^{-1}U_j&\text{if } \epsilon_{j-1}=\text{R},~\epsilon_j\neq\text{L}\\
				U_j&\text{if } \epsilon_{j-1}\neq\text{R},~\epsilon_j\neq\text{L}.
	\end{cases}
\end{align}

We will use the following consequences of the definition \eqref{hatU} in the sequel.

\begin{lem}
If $\epsilon_0,\ldots,\epsilon_{m-1}\in\{\textnormal{L},\textnormal{R}\}$ and $\epsilon_m=0$, then
\label{hatUc1}
$$\check{U}^\epsilon_0\cdots\check{U}^\epsilon_m=U_0(H_0-i)^{-1}U_1\cdots (H_{m-1}-i)^{-1}U_m.$$
\end{lem}

\begin{lem}
\label{hatUc2}
If $J\subseteq\{1,\ldots,m\}$ is such that $|j-j'|\geq2$ for all distinct $j,j'\in J$, then there exists $\epsilon\in\{\textnormal{L},\textnormal{R}\}^{m+1}$ such that $\epsilon_0=\textnormal{R}$, $\epsilon_m=\textnormal{L}$, and
\begin{align*}
\check{U}^\epsilon_j=(H_{j-1}-i)^{-1}U_j(H_j-i)^{-1}\quad\text{ for every }j\in J.
\end{align*}
\end{lem}

We also set
\begin{align}
\label{checkUij}
\check{U}^\epsilon_{i,j}=\check{U}^\epsilon_{i+1}\cdots\check{U}^\epsilon_j.
\end{align}

Denote $\check{U}_{0,0}^{\epsilon}=\check U_{m,m}^\epsilon=I$ for all $\epsilon\in\{\textnormal{L},0,\textnormal{R}\}^{m+1}$ and $\epsilon(i)=\epsilon_i$ for $i=0,\dots,m$. The following theorem, along with its corollaries, forms the first stage of our change of variables method.

\begin{thm}\label{thm:generalized change of variables}
Let $m\in\N$, let $H_0,\dots,H_m$ be self-adjoint in $\H$
and let $V_1,\dots,V_m\in\mB(\H)$. Assume that $\epsilon=(\epsilon_0,\ldots,\epsilon_m)\in\{\textnormal{L},0,\textnormal{R}\}^{m+1}$ with $\epsilon_0\neq\textnormal{L}$ and $\epsilon_m\neq\textnormal{R}$. Then, the following assertions hold.
\begin{enumerate}[(i)]
\item\label{covi} If $\epsilon^{-1}(\{0\})=\emptyset$, then for all $g\in C^m$ such that $\widehat{g^{(m)}},\widehat{(gu^{k+1})^{(k)}}\in L^1(\R)$, $k=0,\dots,m$, we have
\begin{align*}
&T^{H_0,\ldots,H_m}_{g^{[m]}}(V_1,\ldots,V_m)\\
		&=\sum_{k=0}^m(-1)^{m-k}\sum_{\substack{{i_0},\ldots,i_k\in\{0,\ldots,m\}\\
{i_0}<\cdots<i_k}}\check V_{0,{i_0}}^\epsilon T^{H_{i_0},\ldots,H_{i_k}}_{(gu^{k+1})^{[k]}}(\check V_{i_0,i_1}^\epsilon,\ldots,\check V_{i_{k-1},i_k}^\epsilon)\check V_{i_k,m}^\epsilon.
\end{align*}

\item\label{covii} More generally, if $q=|\epsilon^{-1}(\{0\})|$, then for all $g\in C^m$ such that $\widehat{g^{(m)}},\widehat{(gu^{k-q+1})^{(k)}}\in L^1(\R)$, $k=\max(0,q-1),\dots,m$, we have
\begin{align}
\label{summax1q}
\nonumber
		&T^{H_0,\ldots,H_m}_{g^{[m]}}(V_1,\ldots,V_m)\\
		&=\sum_{k=\max(0,q-1)}^{m}(-1)^{m-k}\sum_{\substack{i_0,\ldots,i_k\in\{0,\ldots,m\}\\ i_0<\cdots<i_k\\\epsilon^{-1}(\{0\})\subseteq\{i_0,\ldots,i_k\}}}\check V_{0,i_0}^\epsilon T^{H_{i_0},\ldots,H_{i_k}}_{(gu^{k-q+1})^{[k]}}(\check V_{i_0,i_1}^\epsilon,\ldots,\check V_{i_{k-1},i_k}^\epsilon)\check V_{i_k,m}^\epsilon.
	\end{align}
\end{enumerate}
\end{thm}

\begin{proof}
Since $u^{[1]}=1_{\R^2}$ and $u^{[p]}=0$ for all $p\geq2$, the Leibniz rule for divided differences gives
$$(gu)^{[m]}(\lambda_0,\ldots,\lambda_m)=g^{[m]}(\lambda_0,\ldots,\lambda_m)u(\lambda_m)
+g^{[m-1]}(\lambda_0,\ldots,\lambda_{m-1}).$$
If we swap $\lambda_m$ with $\lambda_j$ (for any $j\in\{0,\ldots,m\}$), and rearrange using symmetry of the divided difference, we obtain
\begin{align}\label{eq:adding one weight}
g^{[m]}(\lambda_0,\ldots,\lambda_m)
=&(gu)^{[m]}(\lambda_0,\ldots,\lambda_m)u^{-1}(\lambda_j)\\
\nonumber
&-g^{[m-1]}(\lambda_0,\ldots,\lambda_{j-1},\lambda_{j+1},\ldots,\lambda_m)u^{-1}(\lambda_j).
\end{align}
As a consequence of \eqref{eq:adding one weight}, we shall prove the following claim inductively: for all $\epsilon\in\{\textnormal{L},0,\textnormal{R}\}^{m+1}$ with $|\epsilon^{-1}(\{0\})|=q$, we have
\begin{align}\label{eq:cov for symbols}
&g^{[m]}(\lambda_0,\ldots,\lambda_m)\\
&=\sum_{k=\max(0,q-1)}^m(-1)^{m-k}\sum_{\substack{i_0,\ldots,i_k\in\{0,\ldots,m\}\\ i_0<\cdots<i_k\\ \epsilon^{-1}(\{0\})\subseteq\{i_0,\ldots,i_k\}}}(gu^{k-q+1})^{[k]}(\lambda_{i_0},\ldots,\lambda_{i_k})
\prod_{\substack{i\in\{0,\ldots,m\}\\\epsilon_i\neq0}}u^{-1}(\lambda_i).\nonumber
\end{align}
The proof proceeds by induction on $m+1-q$.

Firstly we establish the base of induction for $m+1-q=0$ or, equivalently, $|\epsilon^{-1}(\{0\})|=q=m+1$. In this case $\epsilon=(0,\ldots,0)$ and the sum on the right-hand side of \eqref{eq:cov for symbols} comes down to a single term with $k=m$ and $i_j=j$, $j=0,\dots,m$,
which equals the term on the left-hand side of \eqref{eq:cov for symbols}.

For the induction step, we assume that the formula obtained from \eqref{eq:cov for symbols} by substituting $q'=q+1$ for $q$, where $0\leq q\leq m$, holds for all $\epsilon\in\{\textnormal{L},0,\textnormal{R}\}^{m+1}$ satisfying $|\epsilon^{-1}(\{0\})|=q'$. In the following, however, we will assume that $\epsilon\in\{\textnormal{L},0,\textnormal{R}\}^{m+1}$ is such that $|\epsilon^{-1}(\{0\})|=q$. Let $j_0$ be the smallest integer such that $\epsilon_{j_0}\neq0$. Define $\epsilon'\in\{\textnormal{L},0,\textnormal{R}\}^{m+1}$ by setting $\epsilon'_{j_0}:=0$ and $\epsilon'_j:=\epsilon_j$ for all $j\neq j_0$. We then have $|(\epsilon')^{-1}(\{0\})|=|\epsilon^{-1}(\{0\})|+1=q+1=q'$. Therefore, by the induction hypothesis, we have
\begin{align}
\label{gme1}
\nonumber
&g^{[m]}(\lambda_0,\ldots,\lambda_m)\\
&=\sum_{k=\max(0,q)}^m(-1)^{m-k}\sum_{\substack{i_0,\ldots,i_k\in\{0,\ldots,m\}\\
i_0<\cdots<i_k\\ (\epsilon')^{-1}(\{0\})\subseteq\{i_0,\ldots,i_k\}}}(gu^{k-q})^{[k]}
(\lambda_{i_0},\ldots,\lambda_{i_k})\prod_{\substack{i\in\{0,\ldots,m\}\\
\epsilon_i'\neq0}}u^{-1}(\lambda_i).
\nonumber\\
&=\sum_{k=\max(0,q)}^{m}(-1)^{m-k}\sum_{\substack{i_0,\ldots,i_k\in\{0,\ldots,m\}\\ i_0<\cdots<i_k\\ \{j_0\}\cup\epsilon^{-1}(\{0\})\subseteq\{i_0,\ldots,i_k\}}}(gu^{k-q})^{[k]}(\lambda_{i_0},\ldots,\lambda_{i_k})\prod_{\substack{i\in\{0,\ldots,m\}\\\epsilon_i\neq0,~i\neq j_0}}u^{-1}(\lambda_i).
\end{align}
Let $k\in\{\max(0,q),\ldots,m\}$ be arbitrary, and let $i_0,\ldots,i_k\in\{0,\ldots,m\}$ be such that $i_0<\cdots<i_k$ and $\{j_0\}\cup\epsilon^{-1}(\{0\})\subseteq\{i_0,\ldots,i_k\}$.

If $k=0$, then $q=0$, $j_0=0$, and $\{j_0\}\cup\epsilon^{-1}(\{0\})\subseteq\{i_0\}$ implies
$i_0=0$. Therefore,
\begin{align*}
(gu^{k-q})^{[k]}(\lambda_{i_0},\ldots,\lambda_{i_k})\prod_{\substack{i\in\{0,\ldots,m\}\\
\epsilon_i\neq0,~i\neq j_0}}u^{-1}(\lambda_i)&=g(\lambda_0)u^{-1}(\lambda_1)\cdots u^{-1}(\lambda_m)\\
&=(gu^{k-q+1})^{[ k]}(\lambda_{i_0},\ldots,\lambda_{i_k})\prod_{\substack{i\in\{0,\ldots,m\}
}}u^{-1}(\lambda_i).
\end{align*}
If $k\geq 1$, then \eqref{eq:adding one weight} applied to $j=j_0$ implies
\begin{align}
\label{gme2}
\nonumber
&(gu^{k-q})^{[ k]}(\lambda_{i_0},\ldots,\lambda_{i_k})\prod_{\substack{i\in\{0,\ldots,m\}\\
\epsilon_i\neq0,~i\neq j_0}}u^{-1}(\lambda_i)\\
&=(gu^{k-q+1})^{[k]}(\lambda_{i_0},\ldots,\lambda_{i_k})\prod_{\substack{i\in\{0,\ldots,m\}\\
\epsilon_i\neq0}}u^{-1}(\lambda_i)\\
\nonumber
&\quad-(gu^{k-q})^{[k- 1]}(\lambda_{i_0},\ldots,\lambda_{j_0-1},\lambda_{j_0+1},\ldots,\lambda_{i_k})\prod_{\substack{i\in\{0,\ldots,m\}\\\epsilon_i\neq0}}u^{-1}(\lambda_i).
\end{align}
Combining \eqref{gme1} and \eqref{gme2} yields
\begin{align*}
&g^{[m]}(\lambda_0,\ldots,\lambda_m)\\
&=\sum_{k=\max(0,q)}^{m}(-1)^{m-k}\sum_{\substack{i_0,\ldots,i_k\in\{0,\ldots,m\}\\ i_0<\cdots<i_k\\ \{j_0\}\cup\epsilon^{-1}(\{0\})\subseteq\{i_0,\ldots,i_k\}}}(gu^{k-q+1})^{[k]}(\lambda_{i_0},\ldots,\lambda_{i_k})\prod_{\substack{i\in\{0,\ldots,m\}\\\epsilon_i\neq0}}u^{-1}(\lambda_i)\nonumber\\
&\quad-\sum_{k=\max(1,q)}^{m}(-1)^{m-k}\sum_{\substack{i_0,\ldots,i_{k-1}\in\{0,\ldots,m\}\\ i_0<\cdots<i_{k-1}\\ \epsilon^{-1}(\{0\})\subseteq\{i_0,\ldots,i_{k-1}\}\\j_0\notin\{i_0,\ldots,i_{k-1}\}}}(gu^{k-q})^{[k-1]}(\lambda_{i_0},\ldots,\lambda_{i_{k-1}})\prod_{\substack{i\in\{0,\ldots,m\}\\\epsilon_i\neq0}}u^{-1}(\lambda_i).\nonumber
\end{align*}

If $q=0$, then the summation in the first term above starts with $k=\max(0,q)=\max(0,q-1)$.
We can let the index $k$ in the first term run from $k=\max(0,q-1)$ to $m$ because if $q\geq1$ and $k=q-1$, then the condition $\{j_0\}\cup\epsilon^{-1}(\{0\})\subseteq\{i_0,\ldots,i_k\}$ cannot hold, implying that the sum over $i_0,\ldots,i_k\in\{0,\ldots,m\}$ such that $i_0<\cdots<i_k$ and $\{j_0\}\cup\epsilon^{-1}(\{0\})\subseteq\{i_0,\ldots,i_k\}$ is void. Similarly, we can let the index $k$ in the second term run from $k=\max(1,q)$ to $m+1$ because the sum over $i_0,\ldots,i_{k-1}$ is void in the case when $k=m+1$. Shifting the index $k$ in the second term by one, so that the sum runs from $k=\max(0,q-1)$ to $m$, yields
\begin{align*}
&g^{[m]}(\lambda_0,\ldots,\lambda_m)\\
&=\sum_{k=\max(0,q-1)}^{m}(-1)^{m-k}\sum_{\substack{i_0,\ldots,i_k\in\{0,\ldots,m\}\\ i_0<\cdots<i_k\\ \{j_0\}\cup\epsilon^{-1}(\{0\})\subseteq\{i_0,\ldots,i_k\}}}(gu^{k-q+1})^{[k]}(\lambda_{i_0},\ldots,\lambda_{i_k})\prod_{\substack{i\in\{0,\ldots,m\}\\\epsilon_i\neq0}}u^{-1}(\lambda_i)\nonumber\\
&\quad+\sum_{k=\max(0,q-1)}^{m}(-1)^{m-k}\sum_{\substack{i_0,\ldots,i_{k}\in\{0,\ldots,m\}\\ i_0<\cdots<i_{k}\\ \epsilon^{-1}(\{0\})\subseteq\{i_0,\ldots,i_{k}\}\\j_0\notin\{i_0,\ldots,i_{k}\}}}(gu^{k-q+1})^{[k]}(\lambda_{i_0},\ldots,\lambda_{i_{k}})\prod_{\substack{i\in\{0,\ldots,m\}\\\epsilon_i\neq0}}u^{-1}(\lambda_i).\nonumber
\end{align*}
Combining the two sums above gives \eqref{eq:cov for symbols}, completing our induction argument.

The result \eqref{covii} of the theorem follows from \eqref{eq:cov for symbols} by applying Theorem \ref{thm:iptp} to all $g\in C^m$ such that $\widehat{g^{(m)}},\widehat{(gu^{k-q+1})^{(k)}}\in L^1(\R)$, $k=\max(0,q-1),\dots,m$, and noting that the respective resolvents can be placed on either side of the comma. The result \eqref{covi} of the theorem follows as a special case of \eqref{covii}.
\end{proof}

We will apply the above formula to create perturbations of the form $(H-i)^{-1}V(H-i)^{-1}$ by positioning resolvents $(H_j-i)^{-1}$ along the arguments of the multiple operator integral.

The following lemma  will be utilized to obtain useful analytical properties of multiple operator integrals.
\begin{cor}
\label{Vbarin MOI}
Let $n\in\N$, let $H_0,\dots,H_{2n-1}$ be self-adjoint in $\H$, and let $V_1,\dots,V_{2n}\in\mB(\H)$. For all $f\in C^{2n-1}$ satisfying $\widehat{f^{(2n-1)}},\widehat{(fu^{p+1})^{(p)}},\in L^1(\R)$, $p=0,\dots,2n-1$, we have that
\begin{align}\label{T_f^[2n-1]}
&T^{H_0,\ldots,H_{2n-1}}_{f^{[2n-1]}}(V_1,\ldots,V_{2n-1})\nonumber\\
&=\sum_{p=0}^{2n-1}\sum_{0\leq i_0<\cdots<i_p\leq 2n-1}(-1)^{p+1}\,
\overline{V}_{0,i_0}T^{H_{ i_0},\ldots,H_{i_p}}_{(fu^{p+1})^{[p]}}
(\overline{V}_{i_0,i_1},\ldots,\overline{V}_{i_{p-1},i_p})\overline{V}_{i_p,2n-1}.
\end{align}
Moreover, for all $f\in C^{2n}$ satisfying $\widehat{f^{(2n)}},\widehat{(fu^p)^{(p)}}\in L^1(\R)$, $p=0,\ldots,2n$, we have that
\begin{align}\label{T_f^[2n]}
&T^{H_0,\ldots,H_{2n}}_{f^{[2n]}}(V_1,\ldots,V_{2n})\nonumber\\
&=\sum_{p=0}^{2n}\sum_{0\leq i_0<\ldots<i_{p-1}\le 2n-1}(-1)^p\,\overline{V}_{0,i_0}T^{H_{i_0},\ldots,H_{i_{p-1}},H_{2n}}_{(fu^p)^{[p]}}
(\overline{V}_{i_0,i_1},\ldots,\overline{V}_{{i_{p-1}},2n}),
\end{align}
in which the $(p=0)$th term is understood as
$\overline{V}_{0,2n}T^{H_{ 2n}}_{f^{[0]}}()=\overline{V}_{0,2n}f(H_{2n})$.
\end{cor}

\begin{proof}
Both statements are consequences of Theorem \ref{thm:generalized change of variables}.
	
For the first statement, we set $m=2n-1$ and
	$\epsilon=(\textnormal{R},\textnormal{L},\ldots,\textnormal{R},\textnormal{L})\in\{\textnormal{L},0,\textnormal{R}\}^{2n}$. This implies $q=0$ and therefore the statement follows by taking $p=k$.

For the second statement, we set $m=2n$ and $\epsilon=(\textnormal{R},\textnormal{L},\ldots,\textnormal{R},\textnormal{L},0)\in\{\textnormal{L},0,\textnormal{R}\}^{2n+1}$. In this case, $q=1$, $\epsilon^{-1}(\{0\})=\{2n\}=\{i_k\}$, and the statement follows by taking $p=k$.
\end{proof}

%

\paragraph{Change of variables to estimate MOI in terms of the sup-norm of its symbol.}\label{sct:extra groundwork}
In the remainder of this section we fix $m\in\N$, self-adjoint operators $H_0,\ldots,H_m$ in $\H$, and operators $U_0,\ldots,U_m\in\mB(\H)$ which typically are perturbations of the form $\oV_{i,j}$ occurring as arguments of multiple operator integrals on the right-hand side of \eqref{T_f^[2n-1]} or \eqref{T_f^[2n]}.

If we also assume for simplicity that $V_1=\ldots=V_n=:V$ and that one of the (equivalently, both) conditions of Lemma \ref{lem:resolvent comparable} hold, we find that $U_j\in\S^{\alpha_j}$ for $\alpha_0,\ldots,\alpha_m\in[1,\infty]$ and $U_0U_1\cdots U_m\in\S^1$.
However, it might happen that $\alpha_j=\infty$, that is, $U_j$ is not compact for some values of $j$. It is therefore not possible to apply Lemma \ref{lem43} in order to bound an expression of the form $\Tr(U_0T_{g^{[m]}}(U_1,\ldots,U_m))$
in terms of $\supnorm{g^{(m)}}$ and obtain an associated integral formula.
Before we can apply Lemma \ref{lem43}, we need to apply Theorem \ref{thm:generalized change of variables} one more time. The latter two applications are encapsulated in Theorem~\ref{thm:Extra expansion step}, which forms the second stage of our change of variables method. The latter theorem posits a mild assumption on the set of indices $j$ for which $\alpha_j=\infty$, which is satisfied by the applications considered in this paper.

In the next auxiliary lemma we perform a change of variables that allows to obtain estimates for traces of
multilinear operator integrals on tuples of noncompact perturbations in terms of the sup-norm of the symbol.

\begin{lem}\label{lem:U's}
Let $m\in\N$, $U_0,\ldots,U_m\in\mB(\H)$ and let $\alpha_0\in[1,\infty]$ and $\alpha_1,\ldots,\alpha_m\in[1,\infty)$ satisfy $\tfrac{1}{\alpha_0}+\ldots+\tfrac{1}{\alpha_m}=1$. Assume that either $U_0=I$ and $H_0=H_m$, or $\alpha_0<\infty$. Let $\epsilon=(\epsilon_0,\ldots,\epsilon_m)\in\{\textnormal{L},\textnormal{R}\}^{m+1}$ be such that $\epsilon_0=\textnormal{R}$, $\epsilon_m=\textnormal{L}$, and $\check{U}^\epsilon_0\in\S^{\alpha_0},\ldots,\check U^\epsilon_m\in\S^{\alpha_m}$.
Then, there exists a complex Radon measure $\mu$ satisfying
\begin{align*}
\norm{\mu}\leq c_{m,\alpha}\nrm{\check{U}_0^\epsilon}{\alpha_0}\cdots\nrm{\check{U}_m^{\epsilon}}{\alpha_m},
\end{align*}
for a constant $c_{m,\alpha}>0$, and
\begin{align*}
	\Tr(U_0T_{g^{[m]}}^{H_0,\ldots,H_m}(U_1,\ldots,U_m))=\int g^{(m)}u^{m+2}d\mu,
\end{align*}
for all $g\in C^m$ satisfying $\widehat{g^{(m)}}\in L^1$, $\widehat{(gu^{k+1})^{(k)}}\in L^1$, $k=0,\ldots,m$, $g^{(k)}u^{k+1}\in C_0$ for $k=0,\dots,m-1,$ and $g^{(m)}u^m\in L^1$. In particular, the result holds for all $g\in \W_{m+2}^m$.
\end{lem}

\begin{proof}
Applying Theorem \ref{thm:generalized change of variables}\eqref{covi} yields
\begin{align}
\label{lem:U's0}
\nonumber
&\Tr\big(U_0T^{H_0,\ldots,H_m}_{g^{[m]}}(U_1,\ldots,U_m)\big)\\
&=\sum_{k=0}^m(-1)^{m-k}\sum_{0\leq i_0<\cdots<i_k \leq m} \Tr\big(\check{U}^\epsilon_{i_k,m}\check{U}^\epsilon_{-1,i_0}
T^{H_{i_0},\ldots,H_{i_k}}_{(gu^{k+1})^{[k]}}(\check{U}^\epsilon_{i_0,i_1},\ldots,
\check{U}_{i_{k-1},i_k}^\epsilon)\big)
\end{align}
for all $g\in\W_{m+1}^m$, where $\check{U}^\epsilon_{i,j}$ is given by \eqref{checkUij}.

Fix an arbitrary sequence $0\leq i_0<\cdots<i_k \leq m$. We define $\tilde\alpha_0=(\frac{1}{\alpha_{i_{k}+1}}+\ldots+\frac{1}{\alpha_{m}}+\frac{1}{\alpha_0}+\ldots
+\frac{1}{\alpha_{i_0}})^{-1}$ and $\tilde\alpha_l=\big(\frac{1}{\alpha_{i_{l-1}+1}}+\ldots+\frac{1}{\alpha_{i_l}}\big)^{-1}$ for all $l=1,\ldots,k$. We obtain $\tilde\alpha_0\in[1,\infty]$, $\tilde\alpha_1,\ldots,\tilde\alpha_k\in[1,\infty)$ and $\tfrac{1}{\tilde\alpha_0}+\ldots+\tfrac{1}{\tilde\alpha_k}=1$.

Furthermore, we have either $\tilde\alpha_0<\infty$ (allowing us to apply Lemma \ref{lem43}\eqref{item:tr bound ii}) or $i_0=0$, $i_k=m$ and $\tilde\alpha_0=\alpha_0=\infty$, in which case we have $\check{U}^\epsilon_{i_k,m}\check{U}^\epsilon_{-1,i_0}=\check U^\epsilon_0=U_0=I,$ and $H_{i_0}=H_0=H_m=H_{i_k}$ (allowing us to apply Lemma \ref{lem43}\eqref{item:tr bound i}).
By \eqref{lem:U's0} and Lemma \ref{lem43}, there exist complex Radon measures $\mu_0,\ldots,\mu_m$ such that
\begin{align}
\label{lem:U's4}
\Tr(U_0T^{H_0,\ldots,H_m}_{g^{[m]}}(U_1,\ldots,U_m))
=\sum_{k=0}^m\int (gu^{k+1})^{(k)}d\mu_k
\end{align}
for all $g\in \W_{m+1}^m$ and
\begin{align*}
\|\mu_k\|\le{\tilde c_{m,\tilde\alpha}}\nrm{\check{U}_0^\epsilon}{\alpha_0}
\cdots\nrm{\check{U}_m^\epsilon}{\alpha_m}.
\end{align*}

By the Leibniz differentiation rule,
\begin{align}
\label{lem:U's2}
(gu^{k+1})^{(k)}=\sum_{l=0}^k c_{k,l}g^{(l)}u^{l+1}
\end{align}
for some numbers $c_{k,l}$. Applying \eqref{lem:U's4} and \eqref{lem:U's2} yields
\begin{align*}
\Tr(U_0T^{H_0,\ldots,H_m}_{g^{[m]}}(U_1,\ldots,U_m))=\sum_{l=0}^m\int g^{(l)}u^{l+1}d\tilde \mu_l
\end{align*}
for some measures $\tilde\mu_l$ with $\norm{\tilde\mu_l}\leq c_{m,\tilde\alpha}\nrm{\check{U}_0^\epsilon}{\alpha_0}
\cdots\nrm{\check{U}_m^\epsilon}{\alpha_m}$ and a constant $c_{m,\tilde\alpha}$.
Since $g^{(k)}u^{k+1}\in C_0$, $k=0,\dots,m-1$, and $g^{(m)}u^m\in L^1$, the result follows by Lemma \ref{lem:partial integration}.
\end{proof}

We now prove the main result of this section, thereby completing the second stage of the change of variables and its applications to estimation and integral representation of traces of multilinear operator integrals.

\begin{thm}\label{thm:Extra expansion step}
Let $n,m\in\N$ and $\alpha_0,\ldots,\alpha_m\in[1,\infty]$ satisfy $\tfrac{1}{\alpha_0}+\ldots+\tfrac{1}{\alpha_m}=1$. Let $U_j\in\S^{\alpha_j}$ and let $J=\{j\in\{ 1,\ldots,m\}: \alpha_j=\infty\}$. Assume that $|i-j|\geq2$ for all distinct $i,j\in J$, and that $(H_{j-1}-i)^{-1}U_j(H_j-i)^{-1}\in\S^n$ for all $j\in J$. Assume furthermore that either $U_0=I$ and $H_0=H_m$, or $\alpha_0<\infty$. Denote $r=\frac{|J|}{n+|J|}\in[0,1)$ and
\begin{align}
\label{pdef}
& p^{J}_{\alpha}(U_0,\ldots,U_m;H_0,\dots,H_m)\\
\nonumber
&=\norm{U_0}^r\cdots\norm{U_m}^r \prod_{j\in J}\nrm{(H_{j-1}-i)^{-1}U_j(H_j-i)^{-1}}{n}^{1-r}\prod_{j\in\{0,\ldots,m\}\setminus J}\nrm{U_j}{\alpha_j}^{1-r},
\end{align}
where $H_{-1}:=H_m$.
Then, there exist a constant $c_{m,\alpha}>0$ and a complex Radon measure $\mu$ satisfying
\begin{align*}
\norm{\mu}\leq c_{m,\alpha}\,p^{J}_{\alpha}(U_0,\ldots,U_m;H_0,\dots,H_m)
\end{align*}
and
\begin{align}\label{eq:integral formula in extra step}
	\Tr(U_0T_{g^{[m]}}^{H_0,\ldots,H_m}(U_1,\ldots,U_m))=\int g^{(m)}u^{m+2}d\mu
\end{align}
for all $g\in \W_{m+2}^m$. 
\end{thm}

\begin{proof}
Since $J$ has a minimal distance of at least 2, by Lemma \ref{hatUc2} we can fix $\epsilon\in\{\text{L},\text{R}\}^{m+1}$ such that $\epsilon_0=\textnormal{R}$, $\epsilon_m=\textnormal{L}$, and $\check{U}^\epsilon_j=(H_{j-1}-i)^{-1}U_j(H_j-i)^{-1}$ for all $j\in J$. Therefore, $\check U^\epsilon_j\in\S^{\overline{\alpha}_j}$ for all $j\in\{1,\ldots,m\}$, where
\begin{align*}
	\overline\alpha_j=\begin{cases}
	tn\quad&\text{if $j\in J$,}\\
	t\alpha_j&\text{if $j\in\{0,\ldots,m\}\setminus J$}
	\end{cases}
\end{align*}
and	$$t=1+\frac{|J|}{n}.$$
By definition, we have $\overline{\alpha}_1,\ldots,\overline{\alpha}_m\in[1,\infty)$ and $\frac{1}{\overline\alpha_0}+\ldots+\frac{1}{\overline\alpha_m}=\frac{|J|}{tn}+\frac1t=1$.
Since $0\notin J$, $\overline{\alpha}_0=t\alpha_0\in[1,\infty]$.
If $\overline{\alpha}_0=\infty$, then $\alpha_0=\infty$, so by assumption we have $U_0=I$ and $H_0=H_m$.
Hence, by Lemma \ref{lem:U's}, we obtain \eqref{eq:integral formula in extra step}, where a complex Radon measure $\mu$ satisfies
\begin{align*}
\norm{\mu}\leq
c_{m,\alpha}\nrm{\check{U}_0^\epsilon}
{\overline\alpha_0}\cdots\nrm{\check{U}_m^{\epsilon}}{\overline\alpha_m}.
\end{align*}

To conclude the proof of Theorem \ref{thm:Extra expansion step}, we need to establish the bound
\begin{align}\label{eq:U epsilon bound to prove}
\nrm{\check{U}_0^\epsilon}{\overline\alpha_0}\cdots\nrm{\check{U}_m^{\epsilon}}{\overline\alpha_m}\leq p^J_{\alpha}(U_0,\ldots,U_m;H_0,\dots,H_m).
\end{align}
If $U_j=0$ for some $j\in\{0,\ldots,m\}$, then the bound \eqref{eq:U epsilon bound to prove} is trivial. Thus, we assume that $U_j\neq 0$ for all $j\in\{0,\ldots,m\}$.

For any $j\in\{0,\ldots,m\}$, assume that $\alpha_j<\infty$. Let $(\sigma_i)_{i=1}^\infty$ be the sequence of singular values of $U_j/\norm{U_j}$.
Since $\overline\alpha_j\geq\alpha_j\geq 1$ and $\sigma_i\leq 1$ for all $i\in\N$, we obtain
\begin{align}
\nrm{\frac{U_j}{\norm{U_j}}}{\overline\alpha_j}^{\overline\alpha_j}=\sum_{i\in\N} \sigma_i^{\overline\alpha_j}\leq\sum_{i\in\N}\sigma_i^{\alpha_j}
=\nrm{\frac{U_j}{\norm{U_j}}}{\alpha_j}^{\alpha_j}.
\end{align}
Therefore, $\nrm{U_j}{\overline\alpha_j}^{\overline\alpha_j}\leq \nrm{U_j}{\alpha_j}^{\alpha_j}\norm{U_j}^{\overline\alpha_j-\alpha_j}$, and hence,
\begin{align}
\label{Ugeps}
\nrm{\check U^\epsilon_j}{\overline\alpha_j}\leq\nrm{U_j}{\overline\alpha_j}
\leq\nrm{U_j}{\alpha_j}^{\frac{\alpha_j}{\overline\alpha_j}}\norm{U_j}^{1-\frac{\alpha_j}{\overline\alpha_j}}.
\end{align}
If $j\in\{0,\ldots,m\}\setminus J$, then it follows from \eqref{Ugeps} that
\begin{align}\label{eq:bounding U epsilons 1}
\nrm{\check U^\epsilon_j}{\overline\alpha_j}\leq \nrm{U_j}{\alpha_j}^{\frac1t}\norm{U_j}^{1-\frac{1}{t}}.
\end{align}

If $j\in J$, so that $\alpha_j=\infty$, then by a similar argument,
\begin{align}
\label{eq:bounding U epsilons 2}
\nonumber
\nrm{\check U_j^\epsilon}{\overline\alpha_j}&=\nrm{(H_{j-1}-i)^{-1}U_j(H_j-i)^{-1}}{tn}\\ &\leq\nrm{(H_{j-1}-1)^{-1}U_j(H_j-i)^{-1}}{n}^{1/t}\norm{U_j}^{1-\frac{1}{t}}.
\end{align}
Combining \eqref{eq:bounding U epsilons 1} and \eqref{eq:bounding U epsilons 2} yields \eqref{eq:U epsilon bound to prove}, completing the proof of the theorem.
\end{proof}

\section{Local trace formula by finite rank approximations}

In this section, we establish the existence of a locally integrable spectral shift function for a resolvent comparable perturbation.

\begin{thm}\label{thm:eta_n}
Let $n\in\N$, $n\ge 2$, and $H$ be self-adjoint in $\H$. Let $V\in \mB(\H)_{\text{sa}}$ be such that $(H-i)^{-1}V(H-i)^{-1}\in\S^n$. For each $m\in\{2n-1,2n\}$ there exists a function $\eta_m\in L^1_\loc$ such that
$$\Tr(R_{m,H,f}(V))=\int_\R f^{(m)}(x)\eta_m(x)\,dx$$
for all $f\in C^{m+1}_c(\R)$.
\end{thm}

The result of Theorem \ref{thm:eta_n} is established via approximations by finite-rank perturbations. The major technical tool for the proof is provided by Proposition \ref{prop:Tr(R^p(V-W)) conv comp}.

First, we need the following approximation lemma extending the result of \cite[Lemma 4.8]{vNS21} to resolvent comparable perturbations.
\begin{lem}\label{lem:V_k compa conv}
Let $n\in\N$ and $H$ be self-adjoint in $\H$. Let $V\in \mB(\H)_{\text{sa}}$ be such that $(H-i)^{-1}V(H-i)^{-1}\in\S^n$. Then there exists a sequence $(V_k)_{k=1}^\infty$ of self-adjoint finite-rank operators converging strongly to $V$ such that $\|V_k\|\le\|V\|$ for every $k\in\N$,
	\begin{align}\label{eq:V_k comparable conv}
		\nrm{(H-i)^{-1}V_k(H-i)^{-1}-(H-i)^{-1}V(H-i)^{-1}}{n}\to0,
	\end{align}
	and
	\begin{align}\label{eq:V_k comparable bound}
		\nrm{(H-i)^{-1}V_k(H-i)^{-1}}{n}\leq 2\nrm{(H-i)^{-1}V(H-i)^{-1}}{n}.
	\end{align}
\end{lem}
\begin{proof}
For every $k\in\N$, denote $P_k=E_H((-k,k))$ and $Q_k=1-P_k$. By the orthogonality $Q_kP_k=P_kQ_k=0$, and the fact that $P_k$ commutes with $(H-i)^{-1}$, we have
$$\left(P_k(H-i)^{-1}+Q_k\right)P_kVP_k\left(P_k(H-i)^{-1}+Q_k\right)=P_k(H-i)^{-1}V(H-i)^{-1}P_k\in\S^n.$$
Since $P_k(H-i)^{-1}+Q_k$ is invertible by the functional calculus, we have $$P_kVP_k\in\S^n.$$
For any $k$, the spectral theorem yields a sequence $(E_l)_{l=1}^\infty$ of finite-rank projections such that $[E_l,P_kVP_k]=0$ for all $l\in\N$ and $\nrm{E_lP_kVP_k-P_kVP_k}{n}\to 0$ as $l\to\infty$. We fix a subsequence $(E_{l_k})_{k=1}^\infty$ such that
		$$\nrm{E_{l_k}P_kVP_k-P_kVP_k}{n}<1/k.$$
The operator $V_k=E_{l_k}P_kVP_k$ is self-adjoint and converges strongly to $V$ by the fact that both $E_{l_k}P_kVP_k-P_kVP_k\to0$ and $P_kVP_k\to V$ strongly. Moreover,
	\begin{align}
		&\nrm{(H-i)^{-1}V_k(H-i)^{-1}-(H-i)^{-1}V(H-i)^{-1}}{n}\nonumber\\
		&\leq \nrm{(H-i)^{-1}V_k(H-i)^{-1}-(H-i)^{-1}P_kVP_k(H-i)^{-1}}{n}\nonumber\\
		&\quad+\nrm{(H-i)^{-1}P_kVP_k(H-i)^{-1}-(H-i)^{-1}V(H-i)^{-1}}{n}\nonumber\\
		&\leq\norm{(H-i)^{-1}}\nrm{E_{l_k}P_kVP_k-P_kVP_k}{n}\norm{(H-i)^{-1}}\nonumber\\
		&\quad+\nrm{P_k(H-i)^{-1}V(H-i)^{-1}P_k-(H-i)^{-1}V(H-i)^{-1}}{n}\to0.
	\end{align}
	We therefore obtain \eqref{eq:V_k comparable conv}. By the reverse triangle equality, \eqref{eq:V_k comparable conv} implies
	\begin{align}
	\nrm{(H-i)^{-1}V_k(H-i)^{-1}}{n}\to\nrm{(H-i)^{-1}V(H-i)^{-1}}{n},
	\end{align}
	and we therefore also obtain \eqref{eq:V_k comparable bound} by passing to the tail of the sequence $(V_k)_{k=1}^\infty$.
\end{proof}

For future reference, we state a quick corollary of the properties of the sequence $(V_k)_{k=1}^\infty$.

\begin{lem}\label{lem:V_k extra property}
Let $n\in\N$ and $H$ be a self-adjoint operator in $\H$. Let $V\in \mB(\H)_{\text{sa}}$ be such that $(H-i)^{-1}V(H-i)^{-1}\in\S^n$. For a sequence $(V_k)_{k=1}^\infty$ given by Lemma \ref{lem:V_k compa conv}, we define $V^{(k)}$ as being either $V$, $V_k$, or $V-V_k$ for all $k$. Similarly, we define the operators $H_1^{(k)}$ and $H_2^{(k)}$ to either equal $H$, $H+V$, or $H+V_k$ for all $k$. (The three choices need to be made independently of $k$.) We then have
$$\lim^{\S^n}_{k\to\infty}(H_1^{(k)}-i)^{-1}V^{(k)}(H_2^{(k)}-i)^{-1}=\bigg(\lim^{\S^n}_{k\to\infty}(H_1^{(k)}-i)^{-1}\bigg)\bigg(\lim^{\textnormal{so*}}_{k\to\infty}V^{(k)}\bigg)\bigg(\lim^{\S^n}_{k\to\infty}(H_2^{(k)}-i)^{-1}\bigg),$$
where all the limits exist.
\end{lem}

\begin{proof}
The proof is a straightforward case by case study of 27 similar cases; we prove only the case where $V^{(k)}=V_k-V$, $H_1^{(k)}=H+V_k$ and $H_2^{(k)}=H+V$ for all $k$.
Applying subsequently Lemma \ref{ric}, H\"{o}lder's inequality, and \eqref{eq:V_k comparable conv} of Lemma \ref{lem:V_k compa conv} yields
\begin{align}
\label{eq:n-norm bound resolvent identity}&\nrm{(H+V_k-i)^{-1}(V_k-V)(H+V-i)^{-1}}{n}\\
\nonumber&\leq(1+\norm{V_k})\nrm{(H-i)^{-1}(V_k-V)(H-i)^{-1}}{n}(1+\norm{V})\\
\nonumber&\leq (1+\norm{V})^2\nrm{(H-i)^{-1}(V_k-V)(H-i)^{-1}}{n}\to0.
\end{align}
Since the expression on the left-hand side of \eqref{eq:n-norm bound resolvent identity} equals $\|(H+V_k-i)^{-1}-(H+V-i)^{-1}\|_n$, we obtain that the $\S^n$-limit of $(H+V_k-i)^{-1}$ exists. Moreover, the so*-limit of $V_k-V$ exists and is 0, and the $\S^n$-limit of $(H+V-i)^{-1}$ trivially exists. The claim follows.
\end{proof}

The following existence of the spectral shift function for a finite-rank perturbation is a particular case of \cite[Theorem 1.1]{PSS13}.
\begin{prop}
\label{rpvketapk}
Let $m\in\N$ and $H$ be a self-adjoint operator in $\H$. Let $(V_k)_{k=1}^\infty$ be a sequence provided by Lemma \ref{lem:V_k compa conv}.
Then, for every $k\in\N$, there exists $\eta_{m,k}\in L^1(\R)$ such that
\begin{align}\label{Tr(R^p)=int SSF}
\Tr(R_{m,H,f}(V_k))=\int_\R f^{(m)}(x)\eta_{m,k}(x)dx
\end{align}
for all $f\in C_c^{m+1}(\R)$.
\end{prop}

The next approximation result is a core technical component in the proof of the existence of the spectral shift function.

\begin{prop}\label{prop:Tr(R^p(V-W)) conv comp}
Let $m\in\N$, $m\ge 3$, and $H$ be a self-adjoint operator in $\H$. Let $V\in \mB(\H)_{\text{sa}}$ be such that $(H-i)^{-1}V(H-i)^{-1}\in\S^{\lceil m/2\rceil}$. Let $(V_k)_{k=1}^\infty$ be a sequence provided by Lemma \ref{lem:V_k compa conv}. Then, for every $k\in\N$, $a>0$, and $f\in C_c^{m+1}(-a,a)$, we have
$$R_{m,H,f}(V)-R_{m,H,f}(V_k)\in\S^1$$ and
$$\lim_{k\to\infty}\sup_{\substack{f\in C^{m+1}_c{(-a,a)}\\
\supnorm{f^{(m)}}\leq 1}}|\Tr(R_{ m,H,f}(V)-R_{m,H,f}(V_k))|=0.$$
\end{prop}
\begin{proof}
Let $f\in C_c^{m+1}(-a,a)$.
By subsequent application of Theorems \ref{rm}, \ref{dm}, and \ref{perturbation} we obtain
\begin{align*}
&R_{m,H,f}(V_k)\\
&=T^{H,H+V_k,H,\ldots,H}_{f^{[m-1]}}(V_k,\ldots,V_k)-T^{H,\ldots,H}_{f^{[m-1]}}(V_k,\ldots,V_k)\\
&=T^{H,H+V,H,\ldots,H}_{f^{[m-1]}}(V_k,\ldots,V_k)
-T^{H,H+V,H+V_k,H,\ldots,H}_{f^{[m]}}(V_k, V-V_k,V_k,\ldots,V_k)\\
&\quad-T^{H,\ldots,H}_{f^{[m-1]}}(V_k,\ldots,V_k)\\
&=T^{H,H+V,H,\ldots,H}_{f^{[m]}}(V_k, V,V_k,\ldots,V_k)
-T^{H,H+V,H+V_k,H,\ldots,H}_{f^{[m]}}(V_k,V-V_k,V_k,\ldots,V_k).
\end{align*}
By Theorem \ref{rm} and the computation above,
\begin{align*}
&R_{m,H,f}(V)-R_{m,H,f}(V_k)\\
&=\Big(T^{H,H+V,H,\ldots,H}_{f^{[m]}}(V,\ldots,V)-T^{H,H+V,H,\ldots,H}_{f^{[m]}}(V_k, V,V_k,\ldots,V_k)\Big)\\
&\quad+T^{H,H+V,H+V_k,H,\ldots,H}_{f^{[m]}}(V_k,V-V_k,V_k,\ldots,V_k).
\end{align*}
By the telescoping technique applied to the first two terms above,
\begin{align*}
R_{m,H,f}(V)-R_{m,H,f}(V_k)
&=T^{H,H+V,H,\ldots,H}_{f^{[m]}}(V-V_k,V,\ldots,V)\\
&\quad+\sum_{j=3}^{m} T_{f^{[ m]}}^{H,H+V,H,\ldots,H}(V_k,V,\overbrace{V_k,\ldots,V_k}^{j-3\textnormal{ times}},V-V_k,V,\ldots,V)\\
&\quad+T^{H,H+V,H+V_k,H,\ldots,H}_{f^{[m]}}(V_k,V-V_k,V_k,\ldots,V_k).
\end{align*}
Hence, $R_{m,H,f}(V)-R_{m,H,f}(V_k)$ is a sum of terms of the form
\begin{align}\label{eq:terms of the form}
T^{H^{(k)}_0,\ldots,H^{(k)}_m}_{f^{[m]}}(V^{(k)}_1,\ldots,V^{(k)}_m)
\end{align}
where $H^{(k)}_i\in\{H,H+V,H+V_k\}$ for $i=0,\ldots,m$ are such that $H^{(k)}_0=H^{(k)}_m$ and $V^{(k)}_i\in \{V,V_k,V-V_k\}$ for $i=1,\ldots,m$ are such that there exists  $\hat{i}\in\{1,\ldots,m\}$ such that $$V^{(k)}_{\hat{i}}=V-V_k.$$
For the remainder of the proof we assume that $m=2n-1$ for some $n\in\N$, $n\ge 2$.

Applying \eqref{T_f^[2n-1]} of Corollary \ref{Vbarin MOI} to \eqref{eq:terms of the form} gives
\begin{align*}
&T^{H^{(k)}_0,\ldots,H^{(k)}_m}_{f^{[m]}}(V^{(k)}_1,\ldots,V^{(k)}_m)\nonumber\\
&=\sum_{p=0}^{2n-1}\sum_{0\leq i_0<\cdots<i_p\leq 2n-1}(-1)^{p+1}\,\overline{V}^{(k)}_{0, i_0}T^{H^{(k)}_{i_0},\ldots,H^{(k)}_{i_p}}_{(fu^{p+1})^{[p]}}
(\overline{V}^{(k)}_{i_0,i_1},\ldots,\overline{V}^{(k)}_{i_{p-1},i_p})\overline{V}^{(k)}_{i_p,2n-1}.
\end{align*}
Taking the trace, it follows that $\Tr(R_{2n-1,H,f}(V)-R_{2n-1,H,f}(V_k))$ is a sum of terms of the form
\begin{align}\label{term with U's}
\Tr(U_0^{(k)}T^{H_{i_0}^{(k)},\ldots,H_{i_p}^{(k)}}_{(fu^{p+1})^{[p]}}
(U_1^{(k)},\ldots,U_p^{(k)}))
\end{align}
for operators $U_0^{(k)}:=\oV^{(k)}_{i_p,2n-1}\oV^{(k)}_{0,i_0}$ and
$U^{(k)}_{j}:=\oV^{(k)}_{i_{j-1},i_j}$ for $j=1,\ldots,p$.
We have $U_0^{(k)}\in\S^{\alpha_0},\ldots,U^{(k)}_p\in\S^{\alpha_p}$ for some $\alpha_0,\ldots,\alpha_p\in [1,\infty]$ with
$\tfrac{1}{\alpha_0}+\ldots+\tfrac{1}{\alpha_p}=1$.
For the remainder of the proof, we fix a term of $\Tr(R_{2n-1,H,f}(V)-R_{2n-1,H,f}(V_k))$ written in the form \eqref{term with U's}, with the intention of showing that such a generic term converges to 0 as $k\to\infty$. This includes fixing $p$ and fixing $i_0,\ldots,i_p$.

If $p=0$, then $V^{(k)}_{\hat{i}}$ occurs as a factor of $U^{(k)}_0$, and hence \eqref{term with U's} converges to 0 as $k\to\infty$ by Lemmas \ref{lem:sequential continuity}, \ref{lem:V_k compa conv}, and \ref{lem:V_k extra property}. We may therefore in the following assume that $p\geq1$. Moreover, without loss of generality, we may assume that either $U_0^{(k)}=I$ and $H_{i_0}^{(k)}=H_{i_p}^{(k)}$, or $\alpha_0<\infty$. Lastly, we may assume that
\begin{align*}
\{j\in\{1,\ldots,p\}:~\alpha_j=\infty\}&=\{j\in\{1,\ldots,p\}:~i_j\text{ is even and }i_{j-1}+1=i_{j}\}\\
&=:J.
\end{align*}
For all $j\in J$ we have $U_j^{(k)}\in\{V,V_k,V-V_k\}$, and hence $(H_{i_{j-1}}^{(k)}-i)^{-1}U^{(k)}_j(H_{i_j}^{(k)}-i)^{-1}\in\S^n$. Since $j\in J$ implies that $i_{j-1}$ is odd and $i_j$ is even, the minimal distance between elements of $J$ is $\geq2$.
By Theorem \ref{thm:Extra expansion step} and Lemma \ref{lem:tiny lem},
\begin{align}\label{eq:bound p is sufficient}
&\Big|\Tr(U_0^{(k)}T^{H_{i_0}^{(k)},\ldots,H_{i_p}^{(k)}}_{(fu^{p+1})^{[ p]}}(U_1^{(k)},\ldots,U^{(k)}_p))\Big|
\nonumber\\
\nonumber
&\leq \tilde{c}_{\alpha}\supnorm{(fu^{p+1})^{(p)}u^{p+2}} p^{J}_{\alpha}(U_0^{(k)},\ldots,U^{(k)}_p;H_{i_0}^{(k)},\ldots,H_{i_p}^{(k)})\\
&\leq \tilde{c}_{\alpha,a}\supnorm{f^{(2n-1)}} p^{J}_{\alpha}(U_0^{(k)},\ldots,U^{(k)}_p;H_{i_0}^{(k)},\ldots,H_{i_p}^{(k)}).
\end{align}
We are therefore left to show that
$p^{J}_{\alpha}(U_0^{(k)},\ldots,U^{(k)}_p;H_{i_0}^{(k)},\ldots,H_{i_p}^{(k)})$ converges to 0 as $k\to\infty$.

By definition of $\hat{i}$ and $V^{(k)}_{\hat{i}}=V-V_k$, one of the following properties holds:
\begin{enumerate}[(i)]
\item $V^{(k)}_{\hat{i}}=\oV^{(k)}_{i_{j-1},i_j}=U^{(k)}_{j}$ for some $j\in J$;
\item $V^{(k)}_{\hat{i}}$ occurs as a factor in the product $U_0^{(k)}=\oV^{(k)}_{i_p,2n-1}\oV^{(k)}_{0,i_0}$;
\item $V^{(k)}_{\hat{i}}$ occurs as a factor in the product
$U^{(k)}_{j}=\oV^{(k)}_{i_{j-1},i_j}$ for some $j\in\{1,\ldots,p\}\setminus J$.
\end{enumerate}

In the first case we use Lemma \ref{lem:V_k extra property}, and obtain $\|(H^{(k)}_{i_{j-1}}-i)^{-1}U_j^{(k)}(H^{(k)}_{i_j}-i)^{-1}\|_n^{1-r}\to 0$ as $k\to\infty$. Therefore, by \eqref{pdef}, $p^{J}_{\alpha}(U_0^{(k)},\ldots,U^{(k)}_p;H_{i_0}^{(k)},\ldots,H_{i_p}^{(k)})\to0$ as $k\to\infty$.

In the second case, by Lemmas \ref{lem:sequential continuity}, \ref{lem:V_k compa conv}, and \ref{lem:V_k extra property} we obtain
$\|U_0^{(k)}\|_{\alpha_0}^{1-r}\to0$ as $k\to\infty$. Hence, by \eqref{pdef}, $p^{J}_{\alpha}(U_0^{(k)},\ldots,U^{(k)}_p;H_{i_0}^{(k)},\ldots,H_{i_p}^{(k)})\to0$ as $k\to\infty$.

In the third case we similarly find that $\|U_j^{(k)}\|_{\alpha_j}^{1-r}\to0$ as $k\to\infty$. Therefore $p^{J}_{\alpha}(U_0^{(k)},\ldots,U^{(k)}_p;H_{i_0}^{(k)},\ldots,H_{i_p}^{(k)})\to0$ as $k\to\infty$.

In all three of the above cases we find
$p^{J}_{\alpha}(U_0^{(k)},\ldots,U^{(k)}_p;H_{i_0}^{(k)},\ldots,H_{i_p}^{(k)})\to0$ as $k\to\infty$, and, therefore, the result follows from \eqref{eq:bound p is sufficient}.

The case $m=2n$ is proved by the same method, with \eqref{T_f^[2n]} applied instead of \eqref{T_f^[2n-1]}.
\end{proof}

The following result is an immediate consequence of Proposition \ref{prop:Tr(R^p(V-W)) conv comp} and the triangle inequality.
\begin{cor}
\label{corRpVkVm}
Let $m\in\N$, $m\ge 3$, and $H$ be a self-adjoint operator in $\H$. Let $V\in \mB(\H)_{\text{sa}}$ be such that $(H-i)^{-1}V(H-i)^{-1}\in\S^{\lceil m/2\rceil}$, and let $(V_k)_{k=1}^\infty$ be a sequence provided by Lemma \ref{lem:V_k compa conv}. Then, for every $k,l\in\N$, $a>0$, and $f\in C_c^{m+1}(-a,a)$, we have
$$R_{m,H,f}(V_k)-R_{m,H,f}(V_l)\in\S^1$$ and
$$\lim_{k,l\rightarrow\infty}\sup_{\substack{f\in C^{m+1}_c{(-a,a)}\\
\supnorm{f^{(m)}}\leq 1}}|\Tr(R_{m,H,f}(V_k)-R_{m,H,f}(V_l))|=0.$$
\end{cor}

We conclude this section by proving the existence of a locally integrable spectral shift function for a general resolvent comparable perturbation. Nice properties of this function  are derived in the next section.

\begin{proof}[Proof of Theorem \ref{thm:eta_n}]
Let $(V_k)_{k=1}^\infty$ and $(\eta_{m,k})_{k=1}^\infty$ be as in Proposition \ref{rpvketapk}. In particular, Proposition \ref{rpvketapk} implies that
\begin{align}\label{eq:Tr(R) and eta_{k,2a}}
\Tr(R_{m,H,f}(V_k))=\int_{-a}^a f^{(m)}(t)\eta_{m,k}(t)dt
\end{align}
for all $a>0$ and all $f\in C^{m+1}_c(-a,a)$.
By \eqref{Tr(R^p)=int SSF} and Corollary \ref{corRpVkVm}, we obtain
\begin{align}\label{eq:Cauchy SSFs}
	\nrm{\eta_{m,k}-\eta_{m,l}}{L^1[-a/2,a/2]}\le&\sup_{\substack{f\in C^{ m+1}_c{(-a,a)}\\\supnorm{f^{(m)}}\leq1}}\left|\int_{-a}^a f^{(m)}(t)(\eta_{m,k}(t)-\eta_{ m,l}(t))\,dt\right|\nonumber\\
	=&\sup_{\substack{f\in C^{m+1}_c{(-a,a)}\\
\supnorm{{f^{(m)}}}\leq1}}\left|\Tr( R_{m,H,f}(V_k)-R_{m,H,f}(V_l))\right|\to 0
\end{align}
as $k,l\rightarrow\infty$.
Hence $(\eta_{m,k})_{k=1}^\infty$ is Cauchy in $L^1[-a/2,a/2]$ for every $a>0$, and is therefore Cauchy in $L^1_\loc$. By completeness of $L^1_\loc$,
there exists $\eta_m\in L^1_\loc$ such that $$\lim_{k\rightarrow\infty}
\nrm{\eta_{m,k}-\eta_{m}}{L^1[-a,a]}=0$$
for every $a>0$.
The result follows upon combining the latter statement with \eqref{eq:Tr(R) and eta_{k,2a}} and Proposition~\ref{prop:Tr(R^p(V-W)) conv comp}.
\end{proof}

\section{Proof of the main theorem in the odd case}

In this section we derive regularity and uniqueness properties of the spectral shift function whose existence was established in Theorem \ref{thm:eta_n} and extend the class of functions $f$ for which the respective trace formula holds. We attain our goals by establishing the existence of a certain well-behaved measure and showing that its density with respect to the Lebesgue measure coincides with the spectral shift function up to a low-degree polynomial.

\begin{thm}
\label{main}
Let $n\in\N$, $n\ge 2$,  $H$ be self-adjoint in $\H$, and let $V\in\mB(\H)_{\text{sa}}$ be such that $(H+V-i)^{-1}-(H-i)^{-1}\in\S^n$.
Then, there exists a real-valued function $\eta_{2n-1}\in L^1(\R,u^{-4n-2}(x)dx)$ such that
\begin{align}
\label{tf2n-1}
\Tr(R_{2n-1,H,f}(V))=\int f^{(2n-1)}(x)\eta_{2n-1}(x)dx
\end{align}
for every $f\in\W^{2n-1}_{4n+2}$ and such that
\begin{align}\label{eq:main bound SSF}
\int\frac{|\eta_{2n-1}(x)|}{(1+|x|)^{4n+2}}\,dx\leq c_{n}(1+\norm{V}^{2})\norm{V}^{n-1}\nrm{(H-i)^{-1}V(H-i)^{-1}}{n}^{n}
\end{align}
for a constant $c_{n}>0$.
The locally integrable function $\eta_{2n-1}$ is determined by \eqref{tf2n-1} uniquely up to a polynomial summand of degree at most $2n-2$.
\end{thm}

\label{sct:pf main thm}
We denote $H_0=H$, $H_1=H+V$, $H_2=\ldots=H_{2n-1}=H$ and $V_1=\cdots=V_{2n-1}=V$, which along with the notation \eqref{bardef} implies
\begin{align*}
&\oV_1=(H-i)^{-1}V(H+V-i)^{-1},\quad\oV_{2i}=V,\quad \oV_{2i+1}=(H-i)^{-1}V(H-i)^{-1}\quad(i\in\N)
\end{align*}
and $\oV_{i,j}=\oV_{i+1}\cdots\oV_j,\quad \oV_{0,0}=I.$

By Theorem \ref{rm} and the decomposition \eqref{T_f^[2n-1]}, we get
\begin{align}
\label{rpdef}
\nonumber
R_{2n-1,H,f}(V)&=T^{H,H+V,H,\ldots,H}_{f^{[2n-1]}}(V,\ldots,V)\\
\nonumber
&=\sum_{p=0}^{2n-1}(-1)^{p+1}\sum_{0\leq i_0<\cdots<i_p\leq 2n-1}
\oV_{0,i_0}T^{H_{i_0},\ldots,H_{i_p}}_{(fu^{p+1})^{[p]}}(\oV_{i_0,i_1},
\ldots,\oV_{i_{p-1},i_p})\oV_{i_p,2n-1}\\
&=:\sum_{p=0}^{2n-1}(-1)^{p+1}R^p_{2n-1,H,f}(V).
\end{align}

\begin{lem}\label{lem:measure for oV's}
Let $n\in\N$, $n\geq2$. Let $H$ be self-adjoint in $\H$ and let $V\in\mB(\H)_\text{sa}$ satisfy $(H+V-i)^{-1}-(H-i)^{-1}\in\S^n$. For all $p\in\{0,\ldots,2n-1\}$, there exists a complex Radon measure $\mu_{n,p}$ such that
\begin{align}\label{eq:trace formula Rp}
\Tr(R^p_{2n-1,H,f}(V))=\int (fu^{p+1})^{(p)}u^{p+2}\,d\mu_{n,p}
\end{align}
for all $f\in\W^p_{2p+3}$, and such that
\begin{align}
\label{wmunp}
\norm{\mu_{n,p}}\leq c_{n,p}(1+\norm{V}^{2})\norm{V}^{n-1}\nrm{(H-i)^{-1}V(H-i)^{-1}}{n}^{n}.
\end{align}
\end{lem}

\begin{proof}
We fix $i=(i_0,\ldots,i_p)$ satisfying $0\leq i_0<\cdots<i_p\leq 2n-1$
and denote $U_0=\oV_{i_p,2n-1}\oV_{0,i_0}$ and $U_j=\oV_{i_{j-1},i_j}$ for all $j=1,\dots,p$.
By cyclicity of the trace,
\begin{align*}
\Tr\big(\oV_{0,i_0}T^{H_{i_0},\ldots,H_{i_p}}_{(fu^{p+1})^{[p]}}
(\oV_{i_0,i_1},\ldots,\oV_{i_{p-1},i_p})\oV_{i_p,2n-1}\big)=\Tr\big(U_0T^{H_{ i_0},\ldots,H_{i_p}}_{(fu^{p+1})^{[p]}}(U_1,\ldots,U_p)\big).
\end{align*}

We define $\alpha_0:=n(\lceil\frac{{i_0}}{2}\rceil+n-\lceil\frac{i_p}{2}\rceil)^{-1}$ and $\alpha_j:=n(\lceil\frac{i_j}{2}\rceil-\lceil\frac{i_{j-1}}{2}\rceil)^{-1}$ for $j=1,\ldots,p$. We obtain $\tfrac{1}{\alpha_0}+\ldots+\tfrac{1}{\alpha_p}=1$, $U_j\in\S^{\alpha_j}$ for $j=0,\ldots,p$, and find that the space
\begin{align*}
	J:=\{j\in\{1,\ldots,p\}:~\alpha_j=\infty\}=\{j\in\{1,\ldots,p\}~|~\text{$i_j$ is even and } i_{j}=i_{j-1}+1\}
\end{align*}
has minimal distance $\geq2$.
Furthermore, for all $j\in J$ we have $U_j=\overline V_{i_{j-1},i_j}=\overline V_{i_j}=V$ and therefore, by the second resolvent identity,
\begin{align*}
(H_{i_{j-1}}-i)^{-1}U_j(H_{i_j}-i)^{-1}=(H_{i_{j-1}}-i)^{-1}V(H_{i_j}-i)^{-1}\in\S^n.
\end{align*}
Moreover, we have $U_0=\oV_{i_p,2n-1}\oV_{0,i_0}=I$ and $H_{i_0}=H=H_{i_p}$ (since $n\geq2$) whenever $i_p=2n-1,$ $i_0=0$, and $\alpha_0=n(\lceil\frac{{i_0}}{2}\rceil+n-\lceil\frac{i_p}{2}\rceil)^{-1}<\infty$ otherwise.
Therefore all assumptions of Theorem \ref{thm:Extra expansion step} are satisfied. By applying Theorem \ref{thm:Extra expansion step} to $g=fu^{p+1}$ and $m=p$, we obtain a complex Radon measure $\mu_{n,p,i}$ such that
\begin{align}\label{eq:trace formula n,p,i}
\Tr\big(U_0T^{H_{i_0},\ldots,H_{i_p}}_{(fu^{p+1})^{[p]}}(U_1,\ldots,U_p)\big)
=\int(fu^{p+1})^{(p)}u^{p+2}\, d\mu_{n,p,i},
\end{align}
for all $f\in \W^{p}_{2p+3}$ (as this implies that $g=fu^{p+1}\in\W^p_{p+2}$) and a constant $c_{n,p,i}>0$ satisfying
\begin{align}\label{eq:bound n,p,i}
\norm{\mu_{n,p,i}}
&\leq c_{n,p,i} \,p^{J}_{\alpha}(U_0,\ldots,U_p;H_{i_0},\dots,H_{i_p}).
\end{align}
We now bound the factors of $p^{J}_{\alpha}(U_0,\ldots,U_p;H_{i_0},\dots,H_{i_p})$.
Note that
\begin{align*}
\norm{U_0}^r\cdots\norm{U_p}^r\leq \norm{V}^{r(2n-1)}
\end{align*}
and, by Lemma \ref{ric},
\begin{align*}
\prod_{j\in J}\nrm{(H_{i_{j-1}}-i)^{-1}U_j(H_{i_j}-i)^{-1}}{n}^{1-r}
&=\prod_{j\in J}\nrm{(H_{i_{j-1}}-i)^{-1}V(H_{i_j}-i)^{-1}}{n}^{1-r}\\
&\leq (1+\norm{V})\nrm{(H-i)^{-1}V(H-i)^{-1}}{n}^{(1-r)|J|}.
\end{align*}
Note also that if $j\in J$, then $(j+1)\notin J$, $i_j=i_{j-1}+1$, and $\overline{V}_{i_j}=V$ is a factor in none of $U_k$, $k\notin J$. If $j\notin J$, then $U_j=\overline{V}_{i_{j-1},i_j}$ contains exactly
$\lceil\frac{i_j}{2}\rceil-\lceil\frac{i_{j-1}}{2}\rceil$ factors of the form $(H_{k-1}-i)^{-1}V(H_k-i)^{-1}$. Therefore, the operators $\oV_{i_p,2n-1}\oV_{0,i_0}$ and $\oV_{i_{j-1},i_j}$, $j\in\{1,\ldots,p\}\setminus J$, altogether contain as factors all the operators $\oV_i$ for odd $i$ and all but $|J|$ of the operators $\oV_i=V$ for even $i$. Hence, by H\"{o}lder's inequality and definition of $\alpha_k$,
\begin{align*}
\prod_{j\in\{0,\ldots,p\}\setminus J}\nrm{U_j}{\alpha_j}^{1-r}
&=\nrm{\oV_{i_p,2n-1}\oV_{0,i_0}}{\alpha_0}^{1-r}\prod_{j\in\{1,\ldots,p\}\setminus J}\nrm{\oV_{i_{j-1},i_j}}{\alpha_j}^{1-r}\\
&\leq \norm{V}^{(1-r)(n-1-|J|)}(1+\norm{V})\nrm{(H-i)^{-1}V(H-i)^{-1}}{n}^{(1-r)n}.
\end{align*}
To the combination of the last three results we apply the fact that $r=\frac{|J|}{n+|J|}$ and, consequently, $1-r=\frac{n}{n+|J|}$. Hence, by \eqref{pdef}, we obtain
\begin{align}\label{eq:bound pJ n,p,i}
&p^{J}_{\alpha}(U_0,\ldots,U_p;H_{i_0},\dots,H_{i_p})\nonumber\\
&\leq \norm{V}^{r(2n-1)+(1-r)(n-1-|J|)}(1+\norm{V})^2\nrm{(H-i)^{-1}V(H-i)^{-1}}{n}^{(1-r)(|J|+n)}\nonumber\\
&= \norm{V}^{n-1}(1+\norm{V})^{2}\nrm{(H-i)^{-1} V(H-i)^{-1}}{n}^{n}.
\end{align}
By summing the measures $\mu_{n,p,i}$ over $i=(i_0,\dots,i_p)$ we obtain a complex Radon measure $\mu_{n,p}$ such that \eqref{eq:trace formula Rp} follows from \eqref{eq:trace formula n,p,i}. Moreover, \eqref{wmunp} follows from \eqref{eq:bound n,p,i} and \eqref{eq:bound pJ n,p,i}.
\end{proof}

\begin{proof}[Proof of Theorem \ref{main}]
By \eqref{rpdef} and Lemma \ref{lem:measure for oV's}, we obtain
\begin{align}
\label{rpvial}
\Tr(R_{2n-1,H,f}(V))&=\sum_{p=0}^{2n-1}\int (fu^{p+1})^{(p)}u^{p+2}d\mu_{n,p}
\end{align}
for some Radon measures $\mu_{n,p}$ satisfying \eqref{wmunp} and $f\in\bigcap_{p=0}^{2n-1}\W^p_{2p+3}=\W^{2n-1}_{4n+1}$. By applying the higher-order Leibniz rule to \eqref{rpvial},
and applying Lemma \ref{lem:partial integration} to the result, we obtain a Radon measure $\breve\mu_n$ satisfying
\begin{align*}
\Tr(R_{2n-1,H,f}(V))=\int f^{(2n-1)}u^{4n+2}\,d\breve\mu_{n}
\end{align*}
for all $f\in\W^{2n-1}_{4n+2}$ and
\begin{align}\label{eq:almost final bound on SSM}
\norm{\breve\mu_{n}}\leq c_{n,p}(1+\norm{V}^{2})\norm{V}^{n-1}\nrm{(H-i)^{-1}V(H-i)^{-1}}{n}^{n}
\end{align}
(see \eqref{wmunp}).
By defining $d\mu_n:=u^{4n+2}d\breve\mu_n$ we obtain a measure $\mu_n$ such that
\begin{align*}
	\Tr(R_{2n-1,H,f}(V))=\int f^{(2n-1)} \, d\mu_n
\end{align*}
for all $f\in\W^{2n-1}_{4n+2}$ and
\begin{align}\label{eq:final bound on SSM}
\norm{u^{-4n-2}d\mu_{n}}\leq c_{n,p}(1+\norm{V}^{2})\norm{V}^{n-1}\nrm{(H-i)^{-1}V(H-i)^{-1}}{n}^{n}.
\end{align}

Suppose $\tilde\mu_n$ is another measure such that
\begin{align}\label{eq:another SSM}
	\Tr(R_{2n-1,H,f}(V))=\int f^{(2n-1)} \, d\tilde\mu_n
\end{align}
for all $f\in C_c^{2n}$. Then, $\int f^{(2n-1)}d(\tilde\mu_n-\mu_n)=0$ for all $f\in C_c^{2n}$,
so the $(2n-1)$st derivative of the distribution $g\mapsto\int g\,d(\tilde\mu_n-\mu_n)$ equals $0$. Since the primitive of a distribution is unique up to an additive constant,
we obtain that there exists a polynomial $p_{2n-2}$ of degree at most $2n-2$ such that
$$d\tilde\mu_n(x)=d\mu_n(x)+p_{2n-2}(x)dx.$$

The spectral shift function obtained by Theorem \ref{thm:eta_n}, which we shall denote here as $\tilde{\eta}_{2n-1}$ instead of $\eta_{2n-1}$, also induces a measure $\tilde\mu_n$ satisfying \eqref{eq:another SSM}, namely $d\tilde\mu_n(x):=\tilde\eta_{2n-1}(x)dx$. By the argument given above, we find a polynomial $p_{2n-2}$ of degree at most $2n-2$ such that
$$(\tilde\eta_{2n-1}(x)-p_{2n-2}(x))dx=d\mu_n(x).$$
By defining $\acute\eta_{2n-1}:=\tilde\eta_{2n-1}-p_{2n-2}$, we find $\acute\eta_{2n-1}\in L^1_\loc$,
\begin{align}\label{eq:SSF bound almost done}
\absnorm{u^{-4n-2}\acute\eta_{2n-1}}=\norm{u^{-4n-2}d\mu_n}\leq c_{n,p}(1+\norm{V}^{2})\norm{V}^{n-1}\nrm{(H-i)^{-1}V(H-i)^{-1}}{n}^{n},
\end{align}
and
\begin{align}\label{eq:SSF formula almost done}
\Tr(R_{2n-1,H,f}(V))=\int f^{(2n-1)}\acute\eta_{2n-1}+\int f^{(2n-1)}p_{2n-2}=\int f^{(2n-1)}\acute\eta_{2n-1}.
\end{align}
The second inequality in \eqref{eq:SSF formula almost done} follows by $(2n-1)$-times partial integration applied to $\int f^{(2n-1)}p_{2n-2}$ and Proposition \ref{prop:Wsn}. By taking the real part of \eqref{eq:SSF formula almost done}, defining
$$\eta_{2n-1}:=\text{Real}(\acute\eta_{2n-1}),$$ and noting that the left-hand side of \eqref{eq:SSF formula almost done} is real for all real-valued $f\in C^{2n}_c$, we obtain \eqref{tf2n-1}.
Since $|\eta_{2n-1}|\leq|\acute\eta_{2n-1}|$, \eqref{eq:SSF bound almost done} implies that $\eta_{2n-1}\in L^1(\R,u^{-4n-2}(x)dx)$, along with the desired bound \eqref{eq:main bound SSF}.
\end{proof}

\section{Proof of the main theorem in the even case}

In this section we prove the desired properties of the spectral shift function $\eta_{2n}$.
\begin{thm}
\label{main even}
Let $n\in\N$, $n\ge 2$,  $H$ be self-adjoint in $\H$, and let $V\in\mB(\H)_{\text{sa}}$ be such that $(H+V-i)^{-1}-(H-i)^{-1}\in\S^n$.
Then, there exists a real-valued function $\eta_{2n}\in L^1(\R,u^{-4n-3}(x)dx)$ such that
\begin{align}\label{tf2n-1 even}
\Tr(R_{2n,H,f}(V))=\int f^{(2n)}(x)\eta_{2n}(x)dx
\end{align}
for every $f\in\W^{2n}_{4n+3}$ and such that
\begin{align*}
\int\frac{|\eta_{2n}(x)|}{(1+|x|)^{4n+3}}\,dx\leq c_{n}(1+\norm{V}^{2})\norm{V}^{n}\nrm{(H-i)^{-1}V(H-i)^{-1}}{n}^{n}
\end{align*}
for a constant $c_{n}>0$.
The locally integrable function $\eta_{2n}$ is determined by \eqref{tf2n-1 even} uniquely up to a polynomial summand of degree at most $2n-1$.
\end{thm}

We denote $H_0=H$, $H_1=H+V$, $H_2=\ldots=H_{2n}=H$ and $V_1=\cdots=V_{2n}=V$, which along with the notation \eqref{bardef} implies
\begin{align*}
&\oV_1=(H-i)^{-1}V(H+V-i)^{-1},\quad\oV_{2i}=V,\quad \oV_{2i+1}=(H-i)^{-1}V(H-i)^{-1}
\end{align*}
for all relevant $i\in\N$, and $\oV_{i,j}=\oV_{i+1}\cdots\oV_j,\quad \oV_{0,0}=I.$

For simplicity of formulas we define $i_p:=2n$. By Theorem \ref{rm} and the decomposition \eqref{T_f^[2n]}, we obtain
\begin{align}
\label{rpdef even}
\nonumber
R_{2n,H,f}(V)&=T^{H,H+V,H,\ldots,H}_{f^{[2n]}}(V,\ldots,V)\\
\nonumber
&=\sum_{p=0}^{2n}(-1)^{p}\sum_{0\leq i_0<\cdots<i_{p-1}<i_p}
\oV_{0,i_0}T^{H_{i_0},\ldots,H_{i_p}}_{(fu^{p})^{[p]}}(\oV_{i_0,i_1},
\ldots,\oV_{i_{p-1},i_p})\\
&=:\sum_{p=0}^{2n}(-1)^{p}R^p_{2n,H,f}(V).
\end{align}

We obtain the following analogue of Lemma \ref{lem:measure for oV's}.
\begin{lem}\label{lem:measure for oV's even}
Let $n\in\N$, $n\geq2$. Let $H$ be self-adjoint in $\H$ and let $V\in\mB(\H)_\text{sa}$ satisfy $(H+V-i)^{-1}-(H-i)^{-1}\in\S^n$. For all $p\in\{0,\ldots,2n\}$, there exists a complex Radon measure $\mu_{n,p}$ such that
\begin{align}\label{eq:trace formula Rp even}
\Tr(R^p_{2n,H,f}(V))=\int (fu^p)^{(p)}u^{p+2}\,d\mu_{n,p}
\end{align}
for all $f\in\W^{p}_{2p+2}$, and such that
\begin{align}\label{wmunp even}
\norm{\mu_{n,p}}\leq c_{n,p}(1+\norm{V}^{2})\norm{V}^{n}\nrm{(H-i)^{-1}V(H-i)^{-1}}{n}^{n}.
\end{align}
\end{lem}
\begin{proof}
We fix $i=(i_0,\ldots,i_{p-1})$ satisfying $0\leq i_0<\cdots<i_{p-1}<i_p:=2n$.
We denote $U_0=\oV_{0,i_0}$ and $U_j=\oV_{i_{j-1},i_j}$ for all $j=1,\dots,p$. Hence,
\begin{align*}
\Tr\big(\oV_{0,i_0}T^{H_{i_0},\ldots,H_{i_p}}_{(fu^{p})^{[p]}}
(\oV_{i_0,i_1},\ldots,\oV_{i_{p-1},i_p})\big)=\Tr\big(U_0T^{H_{ i_0},\ldots,H_{i_p}}_{(fu^{p})^{[p]}}(U_1,\ldots,U_p)\big).
\end{align*}

We define $\alpha_0:=n(\lceil\frac{{i_0}}{2}\rceil)^{-1}$ and $\alpha_j:=n(\lceil\frac{i_j}{2}\rceil-\lceil\frac{i_{j-1}}{2}\rceil)^{-1}$ for $j=1,\ldots,p$. We obtain $\tfrac{1}{\alpha_0}+\ldots+\tfrac{1}{\alpha_p}=1$, $U_j\in\S^{\alpha_j}$ for $j=0,\ldots,p$, and
\begin{align*}
	J:=\{j\in\{1,\ldots,p\}:~\alpha_j=\infty\}=\{j\in\{1,\ldots,p\}~|~\text{$i_j$ is even and } i_{j}=i_{j-1}+1\}.
\end{align*}
For all $j\in J$ we have $U_j=\overline V_{i_{j-1},i_j}=\overline V_{i_j}=V$ and therefore, by the second resolvent identity,
\begin{align*}
(H_{i_{j-1}}-i)^{-1}U_j(H_{i_j}-i)^{-1}=(H_{i_{j-1}}-i)^{-1}V(H_{i_j}-i)^{-1}\in\S^n.
\end{align*}
Moreover, we have $U_0=\oV_{0,i_0}=I$ and $H_{i_0}=H=H_{i_p}$ whenever $i_0=0$, and $\alpha_0=n(\lceil\frac{{i_0}}{2}\rceil)^{-1}<\infty$ otherwise.
Therefore all assumptions of Theorem \ref{thm:Extra expansion step} are satisfied.
By applying Theorem \ref{thm:Extra expansion step} to $g=fu^{p}$ and $m=p$, we obtain a complex Radon measure $\mu_{n,p,i}$ such that
\begin{align}\label{eq:trace formula n,p,i even}
\Tr\big(U_0T^{H_{i_0},\ldots,H_{i_p}}_{(fu^{p})^{[p]}}(U_1,\ldots,U_p)\big)
=\int(fu^{p})^{(p)}u^{p+2}\, d\mu_{n,p,i},
\end{align}
for all $f\in \W^{p}_{2p+2}$ (as this implies that $g=fu^{p}\in\W^p_{p+2}$) and a constant $c_{n,p,i}>0$ satisfying
\begin{align}\label{eq:bound n,p,i even}
\norm{\mu_{n,p,i}}
&\leq c_{n,p,i} \,p^{J}_{\alpha}(U_0,\ldots,U_p;H_{i_0},\dots,H_{i_p}).
\end{align}

We now bound the factors of $p^{J}_{\alpha}(U_0,\ldots,U_p;H_{i_0},\dots,H_{i_p})$.
Note that
\begin{align*}
\norm{U_0}^r\cdots\norm{U_p}^r\leq \norm{V}^{2nr}
\end{align*}
and by Lemma \ref{ric},
\begin{align*}
\prod_{j\in J}\nrm{(H_{i_{j-1}}-i)^{-1}U_j(H_{i_j}-i)^{-1}}{n}^{1-r}
&=\prod_{j\in J}\nrm{(H_{i_{j-1}}-i)^{-1}V(H_{i_j}-i)^{-1}}{n}^{1-r}\\
&\leq (1+\norm{V})\nrm{(H-i)^{-1}V(H-i)^{-1}}{n}^{(1-r)|J|}.
\end{align*}
Note also that if $j\in J$, then $(j+1)\notin J$, $i_j=i_{j-1}+1$, and $\overline{V}_{i_j}=V$ is a factor in none of $U_k$, $k\notin J$. If $j\notin J$, then $U_j=\overline{V}_{i_{j-1},i_j}$ contains exactly
$\lceil\frac{i_j}{2}\rceil-\lceil\frac{i_{j-1}}{2}\rceil$ factors of the form $(H_{k-1}-i)^{-1}V(H_k-i)^{-1}$. Therefore, the operators $\oV_{0,i_0}$ and $\oV_{i_{j-1},i_j}$, $j\in\{1,\ldots,p\}\setminus J$, altogether contain as factors all the operators $\oV_i$ for odd $i$ and all but $|J|$ of the operators $\oV_i=V$ for even $i$. Hence, by H\"{o}lder's inequality and definition of $\alpha_k$,
\begin{align*}
\prod_{j\in\{0,\ldots,p\}\setminus J}\nrm{U_j}{\alpha_j}^{1-r}
&=\nrm{\oV_{0,i_0}}{\alpha_0}^{1-r}\prod_{j\in\{1,\ldots,p\}\setminus J}\nrm{\oV_{i_{j-1},i_j}}{\alpha_j}^{1-r}\\
&\leq \norm{V}^{(1-r)(n-|J|)}(1+\norm{V})\nrm{(H-i)^{-1}V(H-i)^{-1}}{n}^{(1-r)n}.
\end{align*}
Hence, by \eqref{pdef}, we obtain
\begin{align*}
&p^{J}_{\alpha}(U_0,\ldots,U_p;H_{i_0},\dots,H_{i_p})\\
&\leq \norm{V}^{2nr+(1-r)(n-|J|)}(1+\norm{V})^2\nrm{(H-i)^{-1}V(H-i)^{-1}}{n}^{(1-r)(|J|+n)}.
\end{align*}
Noticing that $r=\frac{|J|}{n+|J|}$ and, consequently, $1-r=\frac{n}{n+|J|}$, we obtain
\begin{align}\label{eq:bound pJ n,p,i even}
p^{J}_{\alpha}(U_0,\ldots,U_p;H_{i_0},\dots,H_{i_p})
\leq\norm{V}^{n}(1+\norm{V})^{2}\nrm{(H-i)^{-1} V(H-i)^{-1}}{n}^{n}.
\end{align}

By summing the measures $\mu_{n,p,i}$ over $i=(i_0,\dots,i_p)$ we obtain a complex Radon measure $\mu_{n,p}$ such that \eqref{eq:trace formula Rp even} follows from \eqref{eq:trace formula n,p,i even}. Moreover, \eqref{wmunp even} follows from \eqref{eq:bound n,p,i even} and \eqref{eq:bound pJ n,p,i even}.
\end{proof}

\begin{proof}[Proof of Theorem \ref{main even}]
By \eqref{rpdef even} and Lemma \ref{lem:measure for oV's even}, we obtain
\begin{align}
\label{rpvial even}
\Tr(R_{2n,H,f}(V))&=\sum_{p=0}^{2n}\int (fu^{p})^{(p)}u^{p+2}d\mu_{n,p}
\end{align}
for some Radon measures $\mu_{n,p}$ satisfying \eqref{wmunp} and $f\in\bigcap_{p=0}^{2n}\W^p_{2p+2}=\W^{2n}_{4n+2}$. By applying the higher-order Leibniz rule to \eqref{rpvial even},
and applying Lemma \ref{lem:partial integration} to the result, we obtain a Radon measure $\breve\mu_n$ satisfying
\begin{align*}
\Tr(R_{2n,H,f}(V))=\int f^{(2n)}u^{4n+3}\,d\breve\mu_{n}
\end{align*}
for all $f\in\W^{2n}_{4n+3}$ and
\begin{align}\label{eq:almost final bound on SSM even}
\norm{\breve\mu_{n}}\leq c_{n,p}(1+\norm{V}^{2})\norm{V}^{n}\nrm{(H-i)^{-1}V(H-i)^{-1}}{n}^{n}
\end{align}
(see \eqref{wmunp even}).
By defining $d\mu_n:=u^{4n+3}d\breve\mu_n$ we obtain a measure $\mu_n$ such that
\begin{align*}
	\Tr(R_{2n,H,f}(V))=\int f^{(2n)} \, d\mu_n
\end{align*}
for all $f\in\W^{2n}_{4n+3}$ and
\begin{align}\label{eq:final bound on SSM even}
\norm{u^{-4n-3}d\mu_{n}}\leq c_{n,p}(1+\norm{V}^{2})\norm{V}^{n}\nrm{(H-i)^{-1}V(H-i)^{-1}}{n}^{n}.
\end{align}

Showing that such a measure is unique up to a polynomial of order $\leq 2n-1$, and subsequently transferring the properties of the measure $\mu_n$ to the spectral shift function $\eta_{2n}$ obtained in Theorem \ref{thm:eta_n}, is done exactly as in the proof of Theorem \ref{main}.
\end{proof}

\paragraph{Acknowledgements} The authors are grateful to Thomas Scheckter for useful comments on an early version of our manuscript. Both A.S. and T.v.N. were supported in part by NSF grant DMS-1554456. In addition, T.v.N. was supported in part by ARC grant FL17010005.


\hspace{4pt}

T.v.N., School of Mathematics and Statistics, University of New South Wales, Kensington, NSW, 2052, Australia
~\\

A.S., Department of Mathematics and Statistics, University of New Mexico, 311 Terrace Street NE, Albuquerque, NM  87106, USA

\begin{thebibliography}{10}
\bibitem{azamov09}
N.~A.~Azamov, A.~L.~Carey, P.~G.~Dodds, F.~A.~Sukochev,
{\it Operator integrals, spectral shift, and spectral flow,}
Canad. J. Math. 61 (2009), no. 2, 241--263.

\bibitem{BP}
M. Sh. Birman, A. B. Pushnitski, {\it Spectral shift function, amazing and multifaceted,} Integral Equations Operator Theory 30 (1998) 191--199.

\bibitem{CS18}
A. Chattopadhyay, A. Skripka, {\it Trace formulas for relative Schatten class perturbations,} J. Funct. Anal. 274 (2018), 3377--3410.

\bibitem{FK}
T. Fack, H. Kosaki, (1986). {\it Generalized $s$-numbers of $\tau$-measurable operators,} Pac. J. Math. 123 (1986), no. 2, 269--300.

\bibitem{GPS}
F.~Gesztesy, A.~Pushnitski, B.~Simon, {\it On the Koplienko spectral shift function. I. Basics,}
Zh. Mat. Fiz. Anal. Geom. {\bf 4} (2008), no. 1, 63--107.

\bibitem{Lifshits}
I.~M. ~Lifshits, {\it On a problem of the theory of pertubations connected with quantum statistics,} Uspekhi Mat. Nauk 7 (1952), no. 1 (47), 171--180 (Russian).

\bibitem{Koplienko84}
L.~S.~Koplienko, \emph{Trace formula for perturbations of nonnuclear type}, Sibirsk. Mat. Zh. 25 (1984), 62--71 (Russian). Translation: Siberian Math. J. 25 (1984), 735--743.

\bibitem{Krein53}
M.~G.~Krein, \emph{On a trace formula in perturbation
theory}, Matem. Sbornik {\bf 33} (1953), 597--626 (Russian).

\bibitem{Krein62}
M.~G.~Krein,
{\it On the perturbation determinant and the trace formula for unitary and self-adjoint operators,}
Dokl. Akad. Nauk SSSR 144 (1962), 268--271 (Russian). Translation: Soviet Math. Dokl. 3 (1962), 707--710.

\bibitem{Neidhardt88}
H.~Neidhardt, {\it Spectral shift function and Hilbert-Schmidt
perturbation: extensions of some work of L.S. Koplienko,} Math. Nachr. 138 (1988), 7--25.

\bibitem{vNS21}
T.~D.~H.~van Nuland, A.~Skripka,
\newblock {\it Spectral shift for relative Schatten class perturbations},
J. Spectr. Theor., to appear, arXiv:2102.00090 [math.FA] (2021).

\bibitem{vNvS21a}
T.~D.~H. van Nuland, W.~D. van Suijlekom, {\it Cyclic cocycles in the spectral action,} J. Noncommut. Geom. Electronically published on December 21, 2021. doi: 10.4171/JNCG/500 (to appear in print).


\bibitem{Peller}
V. V. Peller, {\it Multiple operator integrals and higher operator derivatives,}
J. Funct. Anal. 233 (2006), no. 2, 515--544.

\bibitem{PSS13}
D.~Potapov, A.~Skripka, F.~Sukochev, {\it Spectral shift function of higher order,} Invent. Math. 193 (2013), no. 3, 501--538.

\bibitem{PSS15}
D.~Potapov, A.~Skripka, F.~Sukochev, {\it Trace formulas for resolvent comparable operators,} Adv. Math. 272 (2015), 630--651.




\bibitem{S17}
A.~Skripka, {\it Estimates and trace formulas for unitary and resolvent comparable perturbations}, Adv. Math. 311 (2017), 481--509.


\bibitem{S21}
A.~Skripka, {\it Lipschitz estimates for functions of Dirac and Schr\"{o}dinger operators}, J. Math. Phys. 62, 013506 (2021), no. 1.

\bibitem{ST19}
A. Skripka, A. Tomskova, {\it Multilinear Operator Integrals: Theory and Applications,} Lecture Notes in Math. 2250, Springer International Publishing, 2019, XI+192 pp.

\bibitem{SZ20}
A. Skripka, M. Zinchenko, {\it On uniqueness of higher order spectral shift functions,} Studia Math. 251 (2020), no. 2, 207--218.

\bibitem{YGT}
D. R. Yafaev, {\it Mathematical scattering theory. General theory.} Translations of Mathematical Monographs, 105. American Mathematical Society, Providence, RI, 1992.


\bibitem{Yafaev07}
D. R. Yafaev, {\it The Schrödinger operator: perturbation determinants, the spectral shift function, trace identities, and more.} (Russian) Funktsional. Anal. i Prilozhen. 41 (2007), no. 3, 60--83, 96; Translation: Funct. Anal. Appl. 41 (2007), no. 3, 217--236.

\end{thebibliography}
\end{document}